\documentclass[11pt]{amsart}
\usepackage{amssymb, amsmath, amsthm}
\usepackage{amsrefs}
\usepackage[T1]{fontenc}
\usepackage[utf8]{inputenc}
\usepackage{lmodern}
\usepackage[final]{hyperref}
\usepackage[a4paper, centering]{geometry}
\newtheorem{theorem}{Theorem}[section]
\newtheorem{proposition}[theorem]{Proposition}
\newtheorem{lemma}[theorem]{Lemma}
\newtheorem{corollary}[theorem]{Corollary}
\theoremstyle{definition}
\newtheorem{definition}[theorem]{Definition}
\theoremstyle{remark}
\newtheorem{remark}[theorem]{Remark}

\def\R{\mathbb{R}}

\def\N{\mathbb{N}}

\def\pscal#1#2{\big\langle#1,\,#2\big\rangle}
\def\b{\beta}

\def\hi{m}
\def\dk{\delta_{\hi}}
\def\norm#1{\left\|#1\right\|}

\newcommand{\ut}[1][t]{u_{#1}}

\newcommand{\normh}[3][\Omega]{{\left\| #2 \right\|}_{H^{#3}(#1)}}
\newcommand{\reg}[2][\delta]{\mathcal{C}^{#2}}

\DeclareMathOperator{\esssup}{ess-sup}

\begin{document}

 %amsart format
\title%
{Concavity properties of solutions to Robin problems }

\author[G.~Crasta, I.~Fragal\`a]{Graziano Crasta,  Ilaria Fragal\`a}
\address[Graziano Crasta]{Dipartimento di Matematica ``G.\ Castelnuovo'', Sapienza Universit\`a di Roma\\
P.le A.\ Moro 5 -- 00185 Roma (Italy)}
\email{crasta@mat.uniroma1.it}

\address[Ilaria Fragal\`a]{
Dipartimento di Matematica, Politecnico\\
Piazza Leonardo da Vinci, 32 --20133 Milano (Italy)
}
\email{ilaria.fragala@polimi.it}


 %35E10   	Convexity properties of solutions to PDEs and systems of PDEs with constant coefficients

\date{June 12, 2020}

\begin{abstract} 
We prove that the Robin ground state and the Robin torsion function are respectively log-concave and $\frac{1}{2}$-concave on an uniformly convex domain $\Omega\subset \R ^N$ of class $\mathcal C ^\hi$, with $[\hi -\frac{ N}{2}]\geq 4$, provided
the Robin parameter exceeds a critical threshold. Such threshold depends on  $N$, $m$,  and on the geometry of $\Omega$, precisely on the diameter and on the boundary curvatures
up to order $m$. 
\end{abstract}

\maketitle

\section{Introduction}\label{sec:intro} 

Concavity properties of solutions to elliptic boundary value problems on convex domains have been widely investigated in the literature.
Their study was started in the seventies by Makar-Limanov, who proved the power-concavity of the torsion function in planar domains
\cite{maklim} and then by
 Brascamp-Lieb, who in the pioneering paper \cite{BrLi} established the log-concavity of the Dirichlet ground state via a parabolic approach. 
During the eighties, different methods were developed to deal with more general elliptic equations, 
see the monograph \cite{Kbook} and the references therein. 
In particular, Korevaar invented a new concavity principle \cite{Kor},  
and 
Caffarelli-Friedmann introduced their celebrated method of continuity \cite{CafFri}, then extended in higher dimensions by Korevaar-Lewis \cite{KorLew}. Later,  after the advent of viscosity theory, also the case of fully nonlinear equations has been treated, with  fundamental contributions by
Alvarez-Lasry-Lions \cite{ALL}, Guan-Ma \cite{GuanMa2}, and Caffarelli-Guan-Ma \cite{CaGuMa} (see the survey paper \cite{GuanMa} for more references 
and historical notes). 

A central role is this matter is played by the boundary conditions imposed on the solutions:  to the best of our knowledge, all concavity results available in the literature concern  problems under Dirichlet boundary conditions.  
In quick terms,  the reason is that such conditions allow establishing concavity near the boundary, which is a key step for any among the known methods to work. 
Yet, under Dirichlet boundary conditions, it is possible to go farther and obtain refined concavity estimates, 
such as the  one  given by  Andrews-Clutterbuck  in \cite{AnCl}, with a far-reaching application to the proof of the fundamental gap conjecture;  for different `refinements' of concavity results, see also Ma-Shi-Ye \cite{MaShiYe} and Henrot-Nitsch-Salani-Trombetti \cite{HNST}. 

Very recently, in the ground-breaking paper \cite{AnClHa},  Andrews-Cluttedbuck-Hauer  have attacked   the investigation
of concavity properties under different boundary conditions, of Robin type. The study of the Laplace operator under such kind of boundary conditions is 
a recent trend which is raising an increasing interest in the communities of shape optimization and spectral geometry, see for instance the review papers \cite{BuFrKe, Laug}. 
The discovery in \cite{AnClHa} is that the Robin ground state of a convex set in $\R ^N$  is, in general, not log-concave. More precisely it is proved,  via a perturbation argument, that log-concavity fails
for small (positive) values of the positive Robin parameter on suitable polyhedral domains. 

In the  final section of the paper, the Authors conjecture that the Robin ground state might be log-concave for sufficiently large values of the Robin parameter; 
moreover, they raise the question of understanding the dependence of the  concavity threshold on the space dimension, and possibly on  the geometry of the underlying domain. 

Our paper answers positively to such conjecture, for domains with sufficiently smooth boundary, and provides precise information on the log-concavity threshold. 
We also prove an analogous power-concavity result for the Robin torsion function.  More precisely, 
our main results read as follows.

Let $\Omega\subset \R ^N$ be an open convex bounded domain, 
and  let $\beta$ be a positive real parameter. 

We call {\it Robin ground state} of $\Omega$
a positive solution, normalized so to have unit $L ^ 2$-norm, to 
\begin{equation}
\label{f:robin}
\begin{cases}
-\Delta u = \lambda ^ \b \, u
&\text{in}\ \Omega 
\\
\dfrac{\partial u}{\partial \nu} + \beta\, u = 0 
& \text{on}\ \partial\Omega\,,
\end{cases}
\end{equation}
where $\lambda ^ \b$ is the first Robin eigenvalue of $\Omega$.

We also call  {\it Robin torsion function} of $\Omega$ the unique solution to 
\begin{equation}
\label{f:robin2}
\begin{cases}
-\Delta u = 1 
&\text{in}\ \Omega
\\
\dfrac{\partial u}{\partial \nu} + \beta\, u = 0 
& \text{on}\ \partial\Omega\,.
\end{cases}
\end{equation}

We set the following definitions related to the open set $\Omega$:

  \begin{itemize}

\item[$\cdot$] $\dk (\Omega)$  := 
the sum of the maximum over $\partial \Omega$ of the moduli of all derivatives up to order $\hi$ of a function  representing locally $\partial \Omega$  in a principal coordinate system (provided the latter is of class $\mathcal C ^ m$, see \eqref{def:regu} below for a more detailed definition); 
\smallskip

\item[$\cdot$] $d (\Omega)$ := the diameter of $\Omega$;

\smallskip
\item[$\cdot$] $\kappa_{min} (\Omega) : =\min _{x \in \partial\Omega} \min _{ i = 1, \dots, N-1}  \{ \kappa _i (x)  \} $, 
where $\kappa_1, \ldots \kappa_{N-1}$ are the principal curvatures of $\partial \Omega$ (provided the latter is  of class $\mathcal C ^ 2$). 

\smallskip
\end{itemize}  
Furthermore, for $\gamma \in \R$, we denote by $[\gamma]$ its integer part.

\medskip
We prove: 

\begin{theorem}
\label{thm:concn} Let $\Omega\subset \R ^N$ be an open uniformly convex domain of class $\mathcal C ^\hi$, with 
$[\hi -\frac{ N}{2}]\geq 4$. 
There exists a positive threshold 
$ \b^*$ such that, for  $\b \geq  \b^* $, 
the Robin ground state of $\Omega$  
is strictly log-concave.  Moreover, $\b ^*$ depends only on $N, \hi$, and on the geometry of $\Omega$ through $\delta_ \hi (\Omega)$, $d (\Omega)$ and $\kappa_{min} (\Omega)$, 
with a continuous monotone dependence (increasing in $\delta_ \hi (\Omega)$, $d (\Omega)$, decreasing in $\kappa_{min} (\Omega)$). 
\end{theorem}

\begin{theorem}
\label{thm:concn2}
 Let $\Omega\subset \R ^N$ be an open uniformly convex domain of class $\mathcal C ^\hi$, with 
$[\hi -\frac{ N}{2}]\geq 4$. 
There exists a positive threshold 
$ \b^{**}$ such that, for  $\b \geq  \b^{**} $, 
the Robin torsion function of $\Omega$  
is strictly $1/2$-concave.  Moreover, $\b ^{**}$ depends only $N, \hi$, and on the geometry of $\Omega$ through $\delta_ \hi (\Omega)$, $d (\Omega)$ and $\kappa_{min} (\Omega)$, 
with a continuous monotone dependence (increasing in $\delta_ \hi (\Omega)$, $d (\Omega)$, decreasing in $\kappa_{min} (\Omega)$). 
\end{theorem}

We do not know, at present, whether the assumptions of boundary regularity and uniform convexity are necessary for the validity of the above results. 
For sure, they play a crucial role in our proofs, and it is highly expected  that they may be strictly related to the behaviour of solutions to problems \eqref{f:robin}-\eqref{f:robin2}:  
if few results are known in general under Robin boundary conditions, it is precisely because they are 
much more sensitive to the boundary regularity, and hence  considerably more difficult to handle with respect to Dirichlet ones.

Our proofs are based on the constant rank theorem by Korevaar-Lewis \cite{KorLew}, combined with the continuity method. For its implementation, one needs firstly to know that solutions to problems \eqref{f:robin}-\eqref{f:robin2} on the ball are strictly log-concave, and this is easily checked to be true for any value of the Robin parameter $\beta$. 
But one needs also another crucial information,  namely the fact that, on the domain $\Omega$ under consideration and for $\beta$ large enough, solutions are strictly log-concave in some portion of the domain, typically a neighbourhood of the boundary.  Such property is the heart of the matter. 
The approach we adopt to prove it  
 is inspired by Korevaar \cite{Kor},
and heavily exploits the regularity of Robin solutions and their convergence to the corresponding Dirichlet ones as $\beta$ tends to $+ \infty$. 
Several results are available in this direction in the literature (see \cite{Fil3, Fil4, Fil5, Fil2, Fil1, masrour, Auch} 
for convergence properties and \cite{nitt} about regularity). Nevertheless, none of them covers the $\mathcal C ^2$-convergence that we need for our purposes: roughly speaking, we need that the Hessian matrices of Robin and Dirichlet solutions are uniformly close one to each other.  
For this reason, we need to establish some global regularity estimates for solutions to Robin problems in Sobolev spaces of sufficiently high order 
(which require a boundary regularity of sufficiently high order), 
so that we get the $\mathcal C ^2$-convergence via Morrey-Sobolev embedding theorem. 

The most delicate aspect of this analysis is  the need of tracking the dependence of the log-concavity threshold  on the geometry of the domain, which in particular requires tracking 
all the constants appearing in the regularity and convergence estimates.  We emphasize that tracking the log-concavity threshold is needed 
not only to gain information on it, but to get  its own existence: indeed, without monitoring the
behaviour of the threshold, in principle it might diverge during the deformation process of the ball into a 
given domain $\Omega$  via the method of continuity. For the same reason, the deformation cannot occur in an arbitrary way, 
but must be carefully performed so to keep under control the regularity of the whole family of `intermediate' domains; at this stage,  we take advantage of some 
results for the Minkowski addition of smooth convex bodies given by Ghomi in \cite{ghomi}.  

Let us also point out that it would not be possible to control the log-concavity  threshold without having at disposal some lower bound for the gradient of the 
Dirichlet ground state
(in case of problem \eqref{f:robin}) and the Dirichlet torsion function (in case of problem \eqref{f:robin2}).  
While the latter, in terms of boundary curvatures, is a classical result due to Bandle \cite{bandle},  
the former cannot be found in the literature, so that we needed to set it up. 
For the proof of this result, contained in the Appendix, we are indebted to the kind suggestions of David Jerison, that we warmly acknowledge. 

Finally, let us mention that one of the motivations of the interest by Andrews-Clutterbuck-Hauer 
in studying the log-concavity of the Robin ground state
was its implications in estimating the Robin fundamental gap. In this direction, as
a by-product of our convergence results, we get a lower bound for the Robin fundamental gap of convex sets of class $\mathcal C ^ {1, 1}$, which 
is meaningful for large values of the parameter, see Corollary \ref{thm:gap}. 

\smallskip
In the light of Theorems \ref{thm:concn} and \ref{thm:concn2}, 
we may address the following open questions:

\begin{itemize}
\item[$\cdot$]  Is it possible to remove or weaken the regularity assumptions of Theorems \ref{thm:concn} and \ref{thm:concn2}?

\smallskip
\item[$\cdot$]  Is it possible to characterize  convex sets whose
Robin ground state  (or Robin torsion function)  
is log-concave (resp. $1/2$-concave) for all positive values of $\beta$?    
\end{itemize} 

We are also aware that Theorems \ref{thm:concn} and \ref{thm:concn2} might be extended to more general elliptic operators,
but we preferred to restrict our attention to the Laplacian in order to keep the paper more readable. 

\smallskip
The paper is organized as follows. 
Sections~\ref{sec:estimates} and~\ref{sec:dir} contain respectively the global regularity estimates for solutions of Robin problems and their convergence in Sobolev norms to solutions of Dirichlet problems as 
$\beta \to + \infty$. We point out that such results, which may be of independent interest, are established without asking the convexity of the domain. 

We also warn the reader who is mainly interested in understanding the proof of Theorems~\ref{thm:concn} and~\ref{thm:concn2}, that this is possible by skipping at first reading the technical Sections~\ref{sec:estimates} and~\ref{sec:dir}. 
In fact, in Section~\ref{sec:kor},  before stating 
the crucial concavity property near the boundary (see Proposition \ref{p:Kor}), 
we summarize the required achievements from previous sections.

The proofs of Theorems~\ref{thm:concn} and~\ref{thm:concn2} are given in Section~\ref{sec:proofs}.
Section~\ref{sec:appendix} is devoted to the
boundary gradient estimate needed for the Dirichlet ground state.

Hereafter we fix some notation used throughout the paper.

 %\item{} classical refs:  Brezis 83 \cite{BrezisFr}, Grisvard 85 \cite{Gris},  Mikha\u{\i}lov 78 \cite{Mikha}

\bigskip

{\it Notation.}

For $k \in \N \setminus \{ 0 \}$, we denote by $\lambda _ k ^ \b(\Omega)$ and $\lambda _k ^ D(\Omega)$ respectively 
the Robin and Dirichlet eigenvalues of $\Omega$; 
when $k =1$, we simply write $\lambda ^ \b(\Omega)$ and $\lambda ^ D(\Omega)$.

We  call {\it Dirichlet ground state} of $\Omega$ a first positive eigenfunction
of the Dirichlet Laplacian in $\Omega$,
normalized so to have unit $L ^ 2$-norm.

We call {\it Dirichlet torsion function} the unique solution in $H ^ 1 _0 (\Omega)$ to the equation $- \Delta u = 1$ in $\Omega$.

\smallskip

As usual, 
we say that  $\Omega$ is of class $\mathcal C^{\hi, \alpha}$,
$\hi\in\N$, $\alpha\in [0,1]$, if, for every
$x_0\in\partial\Omega$, there exists  a bijective diffeomorphism $\Psi$ of class $\mathcal C^{\hi, \alpha}$  from
$Q := \{(x', x_N)\in\R^{N-1}\times\R:\
|x'|<1,\ |x_N| < 1\}$  onto an open neighbourhoof $U$ of $x_0$, such that 
$\Psi(Q\cap \{x_N > 0\}) = U\cap \Omega$ and 
$\Psi(Q\cap \{x_N = 0\}) = U \cap \partial\Omega$. Such a map $\Psi$ is called a {\it local chart} around $x_0$. 
In the following, $C^{\hi}$ will stand for $C^{\hi, 0}$.

\smallskip
Let us also recall that, if $\Omega$ is of class $\mathcal C^\hi$, for every $x _0 \in \partial \Omega$ 
{\it the principal coordinate system at $x_0$}  is an orthogonal system   with origin at $x_0$ such that
$\partial \Omega$ can be locally represented, in
a neighborhood of $x_0=0$,
as the graph of a $\mathcal C^\hi$ function
$\varphi _{x_0} \colon B_r' \to \R$, with
$B_r'$ the ball of radius $r$ in $\R^{N-1}$,
$\varphi_{x_0} (0) = 0$, $\nabla\varphi_{x_0} (0) = 0$,
and, if $\hi \geq 2$,
$\partial^2_{ij}\varphi_{x_0} (0) = \kappa_i\, \delta_{ij}$
($i,j = 1, \ldots, N-1$),
where $\kappa_1, \ldots, \kappa_{N-1}$ are the
principal curvatures of $\partial\Omega$ at $0$.

By analogy, given a multi-index $\alpha = (\alpha_1,\ldots,\alpha_{N-1})$
with $|\alpha|:= \alpha_1 + \cdots +\alpha_{N-1} \leq \hi$, we can look at the derivatives 
$\partial^\alpha\varphi_{x_0} (0)$ as higher order curvatures of $\partial \Omega$ at $x_0$.

\smallskip
For any open domain $\Omega \subset \R ^N$ of class $\reg{\hi}$ (with $\hi \in \N$),  we set 
\begin{equation}\label{def:regu}
\delta _\hi (\Omega) :=  \sum _{|\alpha| \leq \hi} \max _{x_0 \in \partial \Omega} |\partial^\alpha\varphi_{x_0}  (0)|\,, 
\end{equation}
where  $\varphi_{x_0}$ is a function which represents locally $\partial \Omega$ around $x_0$ in a principal coordinate system at $x_0$.  We can extend this definition to domains of class
$\mathcal C^{m-1,1}$ (with $m \in N$, $m\geq 1$), by setting
\[
\delta _\hi (\Omega) := 	
\sum _{|\alpha| \leq \hi -1} \max _{x_0 \in \partial \Omega}  |\partial^\alpha\varphi_{x_0}  (0)| +
 \sum _{|\alpha| = \hi} 
\underset{x_0 \in \partial \Omega}{\esssup}\,
|\partial^\alpha\varphi_{x_0}  (0)|\,. 
\]

Note that, by the definition of principal coordinates,
$\delta_0(\Omega) = \delta_1(\Omega) = 0$,
whereas
 $\delta _ 2 (\Omega)$ is  the maximum over $\partial \Omega$ of the sum of the moduli of the principal curvatures. In particular, if $\Omega$ is a convex set of class $\mathcal C^{1,1}$, we have that
\[
\delta_2(\Omega) = \sum_{j=1}^{N-1} \underset{x \in \partial \Omega}{\esssup}\,
\kappa_j(x)\,.
\]

By saying that $\Omega$ is uniformly convex, we mean that all principal curvatures $\kappa_i$ remain uniformly positive along the boundary. 

Furthermore: 
\begin{itemize}

\smallskip 
\item[$\cdot$] $|\Omega|$ denotes  the Lebesgue measure of $\Omega$; 

\smallskip
\item[$\cdot$] $\kappa_{max} (\Omega) : = {\rm ess} \, {\rm sup}_ {x \in \partial \Omega} \max_{i = 1, \dots, N-1}  \{ \kappa _i (x)  \} $, where $\kappa _i $ are the principal curvatures of the boundary of a convex set $\Omega$ of class $\mathcal C ^ {1, 1} $;

\smallskip
\item[$\cdot$] $q (\Omega)$ denotes the minimum over $\partial \Omega$ of the modulus of the gradient
of the Dirichlet ground state;

\smallskip
\item[$\cdot$]  $p (\Omega)$ denotes  the minimum over $\partial \Omega$ of the modulus of the gradient
of the Dirichlet torsion function. 
\end{itemize}

By writing $\Gamma= \Gamma (\alpha _\uparrow, \beta _\downarrow, \ldots)$ we mean that $\Gamma$ is a constant depending continuously  and monotonically on the parameters in parentheses, with increasing monotonicity with respect to $\alpha$, decreasing with respect to   $\beta$, and so forth. 

We  tacitly agree that any of the constants involved in our estimates on a domain $\Omega$ of class $\mathcal C ^ \hi$
in $\R ^N$ may depend on $N$ and $\hi$.

\section{Global regularity estimates}\label{sec:estimates}

In this section we provide uniform upper bounds, in terms of traceable constants independent of $\beta$, 
for the $H ^ \hi $ and $\mathcal C ^ 2$ norms of the Robin ground state and the Robin torsion function.

\begin{theorem}\label{l:bound} (i)  Let
$\Omega\subset\R^N$ be an open domain of class $\reg{\hi }$ (with $\hi  \in \N$). 
Then the Robin ground state $u ^ \b$ on $\Omega$
belongs to $H ^ \hi  (\Omega)$, and satisfies the estimate 
\begin{equation}\label{f:Hk} 
\| u ^ \b  \| _{ H ^ \hi  (\Omega)}  \leq {C} \,,
\end{equation}
for some positive constant $C= C ( \dk(\Omega)_\uparrow, \lambda ^ D (\Omega)_\uparrow)$. 
In particular,  if $[\hi -\frac{N}{2}] \geq 2$, 
$u ^ \b$ belongs to $\mathcal C ^{2, \theta} (\overline \Omega)$ (with $\theta :=  \hi -\frac{N}{2} - [\hi -\frac{N}{2}]$) and 
satisfies, for some positive constant $C$ as above, the estimate 
\begin{equation}\label{f:C2}
\| u ^ \b  \| _{ \mathcal C ^ {2 , \theta} (\overline \Omega)}  \leq   {C}\, . 
\end{equation} 

\smallskip

(ii) Under the same assumptions of the previous item, the estimates \eqref{f:Hk}-\eqref{f:C2} hold as well for the Robin torsion function $u ^ \b$, 
with $C$ replaced by a positive constant $C' = C' (\dk(\Omega) _\uparrow, |\Omega|_\uparrow, \lambda ^ D (\Omega)_\downarrow)$.

\end{theorem}

In order to prove Theorem \ref{l:bound}, we exploit as main tool the following result:

\begin{theorem}
\label{t:reg}
Let  $\Omega$ be an open domain in $\R ^N$,   and let $f\in H^\hi (\Omega)$ (with  $\hi  \in \N$). Let $u$ be the unique
solution to the Robin problem
\begin{equation}
\label{f:robf}
\begin{cases}
-\Delta u + u = f,
&\text{in}\ \Omega,
\\
\dfrac{\partial u}{\partial \nu} + \beta\, u = 0,
& \text{on}\ \partial\Omega\,.
\end{cases}
\end{equation}
If  $\Omega \in \mathcal C^{\hi +2}$, then $u$ 
belongs to $H^{\hi +2}(\Omega)$ and
satisfies
 the estimate
\begin{equation}
\label{f:estir}
\normh{u}{\hi +2} \leq C\, \normh {f}{\hi}\,, 
\end{equation}
for some positive constant $C= C (\delta_{\hi +2}  (\Omega) _\uparrow)$. 

For $\hi  =0$, the estimate \eqref{f:estir}  holds true for any convex domain $\Omega$, regardless its regularity, with 
$C = \sqrt 6$. 
\end{theorem}

Some comments are in order. We collect them in the next remark, and then we proceed by proving first Theorem \ref{t:reg} and then Theorem \ref{l:bound}.

\begin{remark}\label{r:proof-dir}
(i)   In the case $\hi  = 0$ and $\Omega$ convex, the above result is proved in \cite[Theorem 3.2.3.1 and inequality (3.2.3.11)]{Gris}. 
For general $\hi$, results quite similar to Theorem \ref{t:reg} can be found in classical literature about regularity theory  
(see e.g.\ \cite[\S 9.6]{Brezis}, \cite[Chap.~IV, \S 2, Thm.~4]{Mikha}
or \cite[\S 6.4]{GT}). However, to the best of our knowledge, 
the fact that the upper bound 
\eqref{f:estir} holds true for a constant $C = C ( \delta_{m+2}(\Omega) _\uparrow)$ (in particular {\it  independent of $\beta$}) has not been explicitly stated. Since this property is crucial to our purposes, we are going to provide a detailed proof. 

\smallskip
(ii)  The proof of Theorem \ref{t:reg} given below can be readily adapted to obtain  the estimate~\eqref{f:estir}  in the (simpler) case when the Robin boundary condition in \eqref{f:robf} is replaced by the homogeneous Dirichlet one, $u = 0$ on $\partial \Omega$. Consequently, also the estimates \eqref{f:Hk}-\eqref{f:C2}  in Theorem~\ref{l:bound} hold true as well for the Dirichlet ground state, or the Dirichlet torsion function. 
We shall exploit this observation  in the next section (precisely in the proof of Theorem \ref{t:convtoD}). 
 \end{remark}

\begin{proof}[Proof of Theorem \ref{t:reg}] 
By a standard argument (see e.g.\ the first part of the proof of
Theorem~4 in \cite[Chap.~IV, \S 2]{Mikha}, or \cite[\S 9.6-C${}_2$]{Brezis})
it is enough to prove that, when $\Omega$ is replaced by $\R^N_+ := \{x\in\R^N:\ x_N > 0\}$, the inequality  \eqref{f:estir}  holds true 
with $C$ equal to a constant $C_{N, m}$
depending only on $N$ and $m$.
Specifically, assume this result has been proved, and consider an open domain $\Omega\subset\R^N$ of class $\mathcal C^{\hi +2}$.
Using a partition of  unity, we can write $u = \sum_{i=0}^n \zeta_i\, u$,
with $\zeta_i\in \mathcal C^\infty_c(\R^N)$ and with all partial derivatives bounded by 
some constant independent of $\Omega$,
$\zeta_0$ with support in $\Omega$, and $\zeta_1, \ldots, \zeta_n$ with
support in local charts covering $\partial\Omega$. 
Estimate \eqref{f:estir} plainly holds if $u$ is replaced by
$\zeta_0 \, u$ (with a possibly larger constant $C'$ depending only on $N$ and $\hi $).
Consider now $\zeta_i\, u$, for some $i=1,\ldots,n$, with $\zeta_i$ supported
in a local chart around $x_0\in\partial\Omega$.
Using the principal coordinate system at $x_0$, i.e.,
locally parameterizing the boundary of $\partial\Omega$
as the graph of a function $\varphi\in \mathcal C^{\hi +2}$,
it follows that
\eqref{f:estir} holds true for $\zeta_i\, u$ instead of $u$,
with a new constant $C$ depending on $C_{N,m}$ and (increasingly) on $\delta _{\hi +2} (\Omega)$.

\medskip
Let us prove that, when  $\Omega = \R^N_+$, the inequality 
\eqref{f:estir} holds true 
with $C$ equal to a constant $C_{N,m}$ depending only on $N$ and $m$. 
Recall that the unique weak solution $u\in H^1(\Omega)$ of~\eqref{f:robf} is characterized by the weak formulation
\begin{equation}
\label{f:weakh}
\int_\Omega \nabla u \cdot \nabla v \, dx
+ \int_\Omega u \, v \, dx
+ \beta \int_{\partial\Omega} u\, v\, d\sigma =
\int_\Omega f\, v\, dx,
\qquad\forall v\in H^1(\Omega).
\end{equation}

We are going to use tangential translations along
the directions
\[
T := \{h\in\R^N:\ h\neq 0,\ h_N = 0\}.
\]
 
Let  be given $h\in T$. We denote by $\tau_h$ the shift operator
defined by
$(\tau_h v) (x) := v(x+h)$.
Clearly, we have $h + \Omega = \Omega$, and
$v \in H^1(\Omega)$ if and only if $\tau_h v \in H^1(\Omega)$.
Moreover, we denote by $D_h$ the difference quotient operator,
defined by
\[
D_h v := \frac{\tau_h v - v}{|h|}\,,
\qquad\text{i.e.}\qquad
D_h v (x) := \frac{v(x+h) - v(x)}{|h|}\,.
\]

Now, for $h \in T$, let 
us insert $v := D_{-h} D_h u$ as test function in \eqref{f:weakh}.
Observing that
\begin{gather*}
\int_\Omega \nabla u \cdot \nabla(D_{-h} D_h u) \, dx
= \int_\Omega |\nabla D_h u|^2\, dx,
\qquad
\int_\Omega u \, D_{-h} D_h u \, dx
= \int_\Omega (D_h u)^2\, dx,
\\
\int_{\partial\Omega} u\, D_{-h} D_h u\, d\sigma
= \int_{\partial\Omega} (D_h u)^2\, d\sigma\,,
\end{gather*}
we obtain
\[
\int_\Omega |\nabla D_h u|^2\, dx 
+ \int_\Omega (D_h u)^2\, dx
+ \beta \int_{\partial\Omega} (D_h u)^2\, d\sigma
= \int_\Omega f\, D_{-h} D_h u\, dx\,.
\]
Since $\beta \geq 0$, from H\"older's inequality we deduce that
\[
\normh{D_h u}{1}^2 \leq \|f\|_{L^2(\Omega)} \|D_{-h} D_h u\|_{L^2(\Omega)}\,.
\]
Now recall that  
$$
\| D_{-h} v\|_{L^2(\Omega)} \leq
\|\nabla v\|_{L^2(\Omega)},
\qquad
\forall v\in H^1(\Omega),\ \forall h\in T
$$
(see \cite[Lemma 9.6]{Brezis}).  
Using such inequality with $v = D_h u$, we thus conclude that
\begin{equation}
\label{f:55}
\normh{D_h u}{1} \leq \|f\|_{L^2(\Omega)}\,, \qquad 
\forall h \in T.
\end{equation}
Let $j\in \{1,\ldots, N\}$, $i \in \{1,\ldots, N-1\}$, $h = t\, e_i$
($t\in\R$) and let $\varphi \in \mathcal C^\infty_c (\Omega)$.
Integrating by parts and using \eqref{f:55} we have that
\[
\left|\int_\Omega u\, D_{-h}\partial_j \varphi\, dx \right|
= \left|-\int_\Omega \varphi \, D_h \partial_j u \, dx \right|
\leq  \|f\|_{L^2(\Omega)}\, \|\varphi\|_{L^2(\Omega)}\,,
\]
hence, taking the limit as $t\to 0$,
\begin{equation}
\label{f:uij}
\left|\int_\Omega u\, \partial^2_{ij} \varphi\, dx \right|
\leq  \|f\|_{L^2(\Omega)}\, \|\varphi\|_{L^2(\Omega)}\,,
\qquad
\forall j\in \{1,\ldots, N\},\
i \in \{1,\ldots, N-1\}\,.
\end{equation}
From \eqref{f:weakh}, using \eqref{f:uij} and \eqref{f:55} we deduce that
\[
\left|\int_\Omega u\, \partial^2_{NN} \varphi\, dx \right|
\leq
\sum_{j=1}^{N-1} \left|\int_\Omega u\, \partial^2_{jj} \varphi\, dx\right|
+ \left|\int_\Omega (f - u)\varphi \, dx\right|
\leq (N+1)  \|f\|_{L^2(\Omega)}\, \|\varphi\|_{L^2(\Omega)}\,,
\]
hence we conclude that $u\in H^2(\Omega)$ and
\[
\normh{u}{2}^2 = \normh{u}{1}^2
+ \sum_{i,j=1}^N {\| \partial^2_{ij} u\|}_{L_2(\Omega)}^2 \leq
[N (N-1) + (N+1) ^ 2 + 1]  \|f\|_{L^2(\Omega)}^2\,.
\]
This is the required estimate \eqref{f:estir} in the case $\hi=0$, with $C = 2N^2+N+2$. 

\medskip
Let us consider the case $\hi =1$, so that $f\in H^1(\Omega)$.
For every $i\in \{1, \ldots, N-1\}$, we claim that the partial
derivative $\partial_i u$ satisfies
\begin{equation}
\label{f:weakh2}
\int_\Omega \nabla (\partial_i u) \cdot \nabla v \, dx
+ \int_\Omega \partial_i u \, v \, dx
+ \beta \int_{\partial\Omega} \partial_i u\, v\, d\sigma =
\int_\Omega \partial_i f\, v\, dx,
\qquad\forall v\in H^1(\Omega).
\end{equation}
In other words, $\partial_i u$ satisfies the same Robin problem
\eqref{f:robf} as $u$, with source term $\partial_i f$ instead of $f$.
Equation \eqref{f:weakh2} easily follows choosing
$\partial_i v$, with $v\in H^2(\Omega)$, as test function in \eqref{f:weakh},
observing that the maps
$\R^{N-1}\ni x'\mapsto u(x', 0)$,
$\R^{N-1}\ni x'\mapsto v(x', 0)$
belong to $H^{3/2}(\R^{N-1})$, so that
\[
\begin{split}
\int_{\partial\Omega} u\, \partial_i v\, d\sigma
& = \int_{\R^{N-1}} u(x', x_N)\, \partial_i v(x', x_N)\, dx'
= -\int_{\R^{N-1}} \partial_i u(x', x_N)\, v(x', x_N)\, dx' \\
& = - \int_{\partial\Omega} \partial_i u\, v\, d\sigma\,,
\end{split}
\]
and then arguing by density.

By the previous step we deduce that $\partial_i u \in H^2(\Omega)$ and
\[
\normh{\partial_i u}{2}^2 \leq (2N^2+N+2) \|\partial_i f\|_{L^2(\Omega)}^2,
\qquad
i\in \{1, \ldots, N-1\}.
\]
Moreover,
\[
\partial^2_{NN}u = - \sum_{j=1}^{N-1} \partial^2_{jj} u + u - f \in H^1(\Omega)\,,
\]
so that all third derivatives $\partial^3_{ij\ell}u$
belong to $L^2(\Omega)$
and, for a suitable constant $C_{N, 1}$, 
\[
\|\partial^3_{ij\ell}u\|_{L^2(\Omega)}
\leq
 C_{N,1} \, \normh{f}{1}\,,
\qquad
i,j,\ell\in\{1,\ldots,N\}\,. 
\]

\medskip
The general case now follows by induction on $\hi$.
Assume that the claim is true up to order $\hi $, and let us prove
that it holds for $\hi +1$.
By the inductive step we already know that $u\in H^{\hi +2}(\Omega)$ and,
for every $i\in \{1, \ldots, N-1\}$, the partial
derivative $\partial_i u$ satisfies \eqref{f:weakh2}.
Since $\partial_i f \in H^\hi (\Omega)$, we deduce that 
$\partial_i u \in H^{\hi +2}(\Omega)$.
Arguing as above, we also deduce that $\partial^2_{NN} u\in H^{\hi +1}(\Omega)$,
so that all the partial derivatives of order $\hi +3$ belong to $L^2(\Omega)$,
and estimate \eqref{f:estir} (for $\hi +1$) holds.
\end{proof}

\begin{proof}[Proof of Theorem \ref{l:bound}] 
(i) To prove \eqref{f:Hk} we proceed by induction on $\hi $.   From the weak formulation of the problem,
 the Robin ground state $u ^ \b$ satisfies
\[
\int_\Omega \nabla u ^ \b \cdot \nabla v + \b \int _{\partial \Omega} u ^ \b v = \lambda ^ \b \int _{\Omega} u ^ \beta v  \qquad \forall v \in H ^ 1 (\Omega)\,.
\]  
Choosing $v = u ^ \b$,  using the normalization condition in $L ^ 2$  and the inequality $\lambda ^ \b(\Omega) \leq \lambda ^ D(\Omega)$, we get
\[
 \int_\Omega  | \nabla u ^ \b| ^ 2 \leq
\int_\Omega|  \nabla u ^ \b| ^ 2  + \b \int _{\partial \Omega} |u ^ \b| ^ 2  = \lambda ^ \b \int _{\Omega} |u ^ \beta| ^ 2 
\leq \lambda ^ D (\Omega) \,.
\]

We infer that the estimate \eqref{f:Hk} is satisfied for $\hi  = 0$ and $\hi  =1$
(with a monotone increasing dependence of $C$ on $\lambda _ D (\Omega)$).  

To perform the inductive step, assume that \eqref{f:Hk} holds true for a fixed integer $\hi  \in \N$, $\hi \geq 1$. We observe that $u ^ \b$
satisfies 
\[
\begin{cases}
- \Delta u ^ \b + u ^ \b = (\lambda ^ \b + 1 ) u ^ \b & \text { in } \Omega 
\\ \bigskip
\frac{\partial u ^ \b } {\partial \nu} + \b u ^ \b = 0 & \text{ on } \partial \Omega \,. 
\end{cases}
\] 
Then, in view of the assumption $\Omega \in \reg{\hi +2}$, we can apply Theorem \ref{t:reg}  to infer that the 
 the estimate \eqref{f:Hk} is satisfied also for the integer $\hi +2$. 
 
Finally, under the assumption $[\hi -\frac{N}{2}] \geq 2$,  the estimate \eqref{f:C2} follows from \eqref{f:Hk} combined with classical 
Morrey-Sobolev embedding theorem (in which the embedding constant only depends on $N$ and $\delta _\hi  (\Omega)$, 
 see {\it e.g.}\ \cite[Section 9.3]{Brezis}). 
 
 \medskip
 (ii) For the Robin torsion function $u ^ \b$ the proof is the same except that, 
 starting from  the weak formulation 
\[
\int_\Omega \nabla u ^ \b \cdot \nabla v + \b \int _{\partial \Omega} u ^ \b v =  \int _{\Omega}  v  \qquad \forall v \in H ^ 1 (\Omega)
\]  
and choosing $v = u ^ \b$, one gets
\[
 \begin{array}{ll} 
\displaystyle \int_\Omega  | \nabla u ^ \b| ^ 2 & \displaystyle \leq
\int_\Omega|  \nabla u ^ \b| ^ 2  + \b \int _{\partial \Omega} |u ^ \b| ^ 2  =  \int _{\Omega} u ^ \beta  
\\ 
\noalign{\bigskip}
& \displaystyle 
 \leq |\Omega| ^ {1/2} \Big [ \int _{\Omega} |u ^ \beta| ^ 2  \Big ] ^ { 1/2} 
 \leq |\Omega| ^ {1/2}   [\lambda ^ D(\Omega) ]^ { -1/2}  \Big [ \int _{\Omega} |\nabla u ^ \beta| ^ 2  \Big ] ^ { 1/2} 
\,.
\end{array}
\]  
In particular, by the use of H\"older's inequality in the last line above, one sees that $C'$ depends also (increasingly) on $|\Omega|$, while it now depends decreasingly on 
$\lambda ^ D (\Omega)$. 
\end{proof}

\section{Convergence to solutions of Dirichlet problems}\label{sec:dir} 

This section is devoted to prove a convergence result in $H ^m$ of solutions to Robin's boundary value problems to solutions of the corresponding Dirichlet problems.  

In the particular case $\hi = 2$ and without tracking the dependence of the constant $M$ on the domain, the estimate \eqref{f:convHk}  for the ground states has been proved in 
\cite[Theorem 4]{Fil1} (in the two-dimensional setting, see also \cite{masrour}).  Actually, our main objective is getting the estimate \eqref{f:convHk} for higher values of $\hi$,
so to arrive at the convergence in $\mathcal C ^ 2$ in \eqref{f:convC2} that we need specifically in the next section 
(along with a full control on the involved constants).% 
\footnote{Incidentally, let us further mention that we were not able to check the independence from the Robin parameter of the constant $C_{13}$ in the estimate (5.19) in
the proof of \cite[Theorem 4]{Fil1}, since it does not seem clear how to
deal with the boundary integral $I$ that appears in such proof.}

\begin{theorem}\label{t:convtoD}
(i) Let $\Omega \subset \R ^N $ be an open domain   of class $\mathcal C ^ {\hi +2}$, with $\hi  \in \N$. 
For every $\b>0$, let $u ^\b$  and $u ^ D$ be respectively the Robin  and the Dirichlet ground states.  
There exists a  positive constant 
$M= M (\delta_{\hi+2} (\Omega) _\uparrow , d (\Omega)_\uparrow , \lambda ^ D (\Omega)_\uparrow  )$ such that
\begin{equation}\label{f:convHk}
\| u ^ \b - u ^ D \| _{ H ^ \hi (\Omega)}  \leq \frac{M}{\b }\,. 
\end{equation}
In particular,  if $[\hi -\frac{N}{2}] \geq 2$,  and for a positive constant $M$ as above, it holds that
\begin{equation}\label{f:convC2}
\| u ^ \b - u ^ D \| _{ \mathcal C ^ {2, \theta} (\overline \Omega)}  \leq   \frac{M}{\b }\,.
\end{equation}
(ii) Let $\Omega \subset \R ^N $ be an open domain  of class $\mathcal C ^ {\hi+h}$, with $\hi , h \in \N$, $h \geq 2$ and $[h- \frac{N}{2}] \geq 1$.
The estimates \eqref{f:convHk}-\eqref{f:convC2} hold as well when $u ^\b$ and $u ^ D$ are the Robin and the Dirichlet torsion functions, 
with $M$ replaced by a positive constant $M' = M' (\delta_{m+h} (\Omega)_\uparrow , |\Omega |_\uparrow, \lambda ^ D (\Omega)_\downarrow)$. 
\end{theorem}

\medskip
The proof of Theorem \ref{t:convtoD},  given at the end of this section, will be by induction on $\hi$. 

The case $\hi = 0$ is considered separately in Proposition \ref{l:cL2} below. 
We shall make repeated use of the next lemma. 

\medskip

\begin{lemma}\label{l:ext}
Let $\Omega\subset\R^N$ be an open domain of class $\mathcal C^{\hi, 1}$,
$\hi \geq 1$.
Then the unit outer normal vector $\nu$ to $\partial \Omega$ admits an extension $b\in W ^ {m, \infty}(\Omega; \R^N)$
such that 
$$\| b\|_{W^{\hi, \infty} (\Omega; \R ^N)} \leq K\, , \ \text{ with }
K = K(\delta_{\hi+1}(\Omega)_\uparrow) \,.$$ 
\end{lemma}

\begin{proof}
Let  $d$ denote the distance function from $\partial \Omega$. Consider
an inner tubular neighbourhood $U  _\varepsilon:= \{x \in \Omega \, : \, 0 < d(x) < \varepsilon)$. We fix
 $\varepsilon\in (0, 1)$ sufficiently small so that $d$ is of class 
$\mathcal C^{\hi, 1}$ in $U _\varepsilon$
(see \cite[Theorem~6.10]{CMf}), and $\delta_{\hi+1}(U_t ) \leq 2 \, \delta_{\hi+1}(\Omega)$ for every $t \in [0, \varepsilon]$. 
By the implicit function theorem,  we may find a positive constant $K_1 = K_1 (N,m)$ such that, for every multi-index $\alpha$ with $2 \leq |\alpha| \leq m+1$ and every $t \in [0, \varepsilon]$, it holds that
\[
\underset{\{ d(x) = t \}}{\esssup}\, |D^\alpha d|  \leq K _1 \,   \delta _{m+1} ( U _ t) \, .
\] 
Hence,
\begin{equation}\label{f:nei}
\|\nabla d\|_{W^{\hi, \infty}(U_\varepsilon)} \leq  1+ 2 K_1  \, \delta_{\hi+1}(\Omega)\,.
\end{equation}

By Theorems~5 and $5'$ in \cite[VI.\S3]{Stein}, we can extend $d$ from $U_\varepsilon$ to a function 
$\tilde d \in W^{\hi+1, \infty}(\Omega)$ satisfying, for a positive constant $K_2 = K_2 (N, m)$, 
\begin{equation}\label{f:nei2}
\|\tilde d \|_{W^{\hi+1, \infty}(\Omega)}
\leq K_2  \, \|d\|_{W^{\hi+1, \infty}(U_\varepsilon)}\,.
\end{equation}
(Actually, according to \cite[VI.\S3, Theorems~$5'$]{Stein}, the constant $K_2$ should also depend on the Lipschitz constant of $U _\varepsilon$, 
but we have omitted this dependence since we can assume with no loss of generality that $U _\varepsilon$ is $1$-Lipschitz.)

Now, choosing $b = \nabla \tilde d$, the claim follows  
by using \eqref{f:nei} and \eqref{f:nei2}, and recalling that $d<\varepsilon<1$ in $U _\varepsilon$. 
\end{proof}

\begin{proposition}\label{l:cL2}
(i) Let $\Omega\subset \R ^N$ be an open domain of class $\mathcal C ^ {2}$, and let 
$u ^\b, u ^ D$ be respectively the 
Robin and Dirichlet ground states. 

There exist positive constants $\Lambda_k =  \Lambda_k (\delta_2(\Omega)_\uparrow , \lambda_k  ^ D (\Omega)_\uparrow  )$ and $\Lambda= \Lambda (\delta_2(\Omega)_\uparrow , d  (\Omega)_\uparrow   )$, such that
\begin{eqnarray}
|\lambda _k ^ \b (\Omega) - \lambda_k ^ D(\Omega) | \leq \frac{\Lambda_k}{\b}
& \label{f:convL2_uno}
\\
\| u ^ \b - u ^ D \| _{ L ^ 2 (\Omega)}  \leq \frac{\Lambda}{\b} \,.
& \label{f:convL2_due}
\end{eqnarray}

\smallskip
(ii)  Let $h\in \N$, and let $\Omega\subset \R ^N$ be an open domain of class $\mathcal C ^ h$, with
$[h- \frac{N}{2}] \geq 1$.
The estimate in \eqref{f:convL2_due} holds as well when $u ^\b$ and $u ^ D$ are the Robin and the Dirichlet torsion functions, 
with $\Lambda$ replaced by another positive constant $\Lambda' = \Lambda' ( \delta_h(\Omega)_\uparrow  , |\Omega|_\uparrow , \lambda ^ D (\Omega)_\downarrow )$.

\end{proposition}

\begin{remark} Though the asymptotic results stated in Proposition \ref{l:cL2} can be proved more in general on Lipschitz domains, 
we work on $\mathcal C ^2$ domains in order to control  the rate of convergence. Let us also mention that, even if 
 higher order Robin eigenfunctions converge as well to Dirichlet ones, we neglect them since they are not needed in the sequel. 
\end{remark}

\begin{remark}\label{rem:explicit} If $\Omega$ is a convex sets of class $\mathcal C ^ {1,1}$, by inspection of the proof of Proposition \ref{l:cL2} given below (and recalling that for $\Omega$ convex the constant $C$ in Theorem \ref{t:reg} is explicit), one can easily check that the  estimates \eqref{f:convL2_uno} and \eqref{f:convL2_due} hold true with  
\[
\Lambda _k = \sqrt 6  \| b \| _{ W ^ {1, \infty}}  (\lambda_k ^ D + 1)^2  \,, 
\qquad 
\Lambda = \frac{4 \sqrt 3 d(\Omega) ^ 2}{3 \pi ^ 2 }  \| b \| _{ W ^ {1, \infty}}  \,,
\]
where $b\in W ^ {1, \infty}(\Omega; \R ^N)$  is an extension of the unit outer normal vector to $\Omega$ as given by  Lemma \ref{l:ext}. (We precise for later use that here
$\norm{b}_{ W ^ {1, \infty}} 
:= \norm{ |b| }_{L^\infty} + 
\norm{ |D b| }_{L^\infty}$, $|b|$ and $|Db|$ being respectively the Euclidean norm of the vector $b$ and of the matrix $Db$).  

\end{remark}

\begin{proof}[Proof of Proposition \ref{l:cL2}]   
(i)  We follow the same arguments as  in  \cite{Fil1}, where the estimates \eqref{f:convL2_uno}-\eqref{f:convL2_due} are given
  without control of the involved constants. 
  Our main concern is  precisely to track the dependence on the domain of the constants appearing in these estimates.

Note firstly that, given $g \in L ^ 2 (\Omega)$ and 
$\Phi \in H^{1} (\Omega)$,  
the unique solution $z \in H ^ 1 (\Omega)$ to the non-homogeneous Dirichlet boundary value problem
\[
\begin {cases}
- \Delta z  + z= g  & \text { in } \Omega 
\\ 
z = \Phi  & \text{ on } \partial \Omega \,,
\end{cases}
\] 
satisfies the inequality 
\begin{equation}\label{f:base} 
\| z \| _{ H ^ 1 (\Omega)} \leq \| g \| _{ L ^ 2 (\Omega)} + 
\| \Phi \| _{ H^{1} (\Omega)}\,.
\end{equation}
This is readily checked by writing the weak formulation of the problem (see  {\it e.g.}\ \cite[eq. (47) p.196]{Mikha}).

Consider the linear operators  $A ^ \b, A ^ D: L ^ 2 (\Omega) \to L ^ 2 (\Omega)$ defined by $A ^ \b f = u$ and $A ^ D f = v$, with
\[
\begin {cases}
- \Delta u  + u= f  & \text { in } \Omega 
\\
 \frac{\partial u  } {\partial \nu} + \b u = 0  & \text{ on } \partial \Omega \,,
\end{cases}
\qquad
\begin{cases}
- \Delta v  + v= f  & \text { in } \Omega 
\\
v = 0  & \text{ on } \partial \Omega \,.
\end{cases}
\]
In order to prove \eqref{f:convL2_uno}, 
let us determine a  bound from above for  the norm of the difference operator $A^ \b - A ^ D$. We observe that, given $f \in L ^ 2 (\Omega)$, 
and setting $u := A^ \b f$, $v: = A ^ D f$,  the function $w :=  u - v$ satisfies 
\[
\begin {cases}
- \Delta w  + w= 0  & \text { in } \Omega 
\\
w = \frac{1}{\b}  \frac{\partial u  } {\partial \nu}  & \text{ on } \partial \Omega \,.
\end{cases}
\] 
By Lemma \ref{l:ext}, there exists a vector field 
$b \in W ^ {1,  \infty} (\Omega; \R ^n)$, 
 with $b = \nu$ on $\partial \Omega$,  such that
 $\| b\|_{W^{1, \infty} (\Omega; \R ^N)} \leq K(\delta_{2}(\Omega)_\uparrow)$. 
  %$\|b \|_{ L ^ \infty(\Omega)} \leq 1$ and $\| \nabla b \| _{  L ^ \infty (\Omega)} \leq \sqrt {N-1} \,  \kappa_{max}  (\Omega) $ 
  Moreover, by Theorem \ref{t:reg},  applied with $\hi = 0$, we know that $u \in H ^ 2 (\Omega)$, and satisfies  
  $\|u \| _{ H ^ 2 (\Omega)} \leq C ( \delta _ 2 (\Omega) _\uparrow)  \| f \| _{ L ^ 2 (\Omega) }$. 
  
Then, by the inequality  \eqref{f:base} (applied with $z = w$, $g = 0$, and $\Phi = \frac{1}{\b}  \nabla u \cdot b $), we obtain, for a positive constant
$C = C ( \delta _ 2 (\Omega) _\uparrow)$, 
\begin{equation}\label{f:step0} 
\|w\| _{ H ^ 1 (\Omega)} \leq \frac{1}{\beta} \| \nabla u \cdot b \| _{ H ^ 1 (\Omega)} \leq \frac{1}{\b} \| u \| _{ H ^ 2 (\Omega) }\| b\|_{W^{1, \infty} (\Omega; \R ^N)} 
\leq \frac{1}{\beta}  C \| f \| _{ L ^ 2 (\Omega) }
\,.  
\end{equation}

We infer that 
\[
\| (A ^ \b - A ^ D )f \| _{ L ^ 2 (\Omega)} \leq \|w\| _{ H ^ 1 (\Omega)} \leq   \frac {C}{\b}   \|  f   \| _{ L ^ {2} ( \Omega)}  
\,,
\] 
and hence
\begin{equation}\label{f:operators}
\| A ^ \b - A ^ D \| \leq \frac{C}{\beta}\,. 
\end{equation} 
In the remaining of the proof, $C= C (\delta_2 (\Omega) _\uparrow)$ denotes the constant appearing in \eqref{f:operators}. 

\smallskip
Then we observe that the eigenvalues $\mu _k^\b$, $\mu _k^ D$ of the self-adjoint positive compact operators $A ^ \b$, $A ^D$, are related to $\lambda _k^ \b$, $\lambda _k^ D$ by the equalities
\[
\mu _k^ \b = \frac{1}{\lambda_k ^ \b +1} \, , \qquad \mu_k ^ D = \frac{1}{\lambda_k ^ D + 1} \,,
\]
and satisfy the estimate $| \mu _k^ \b - \mu_k ^ D |  \leq \| A ^ \b - A ^ D \|$ (see \cite[Theorem 5]{Fil1}). Hence,  recalling that $\lambda _ k ^ \b \leq \lambda _k ^ D$ (see \cite[Section 2]{Laug}), we have
$$|\lambda _k ^ \b - \lambda _k ^ D| \leq \frac{1}{\beta} {C } (\lambda_k ^ \b + 1) (\lambda_k ^ D + 1)\leq \frac{1}{\beta} {C } (\lambda_k ^ D + 1)^2  \,, $$ 
which proves \eqref{f:convL2_uno} with
$$\Lambda_k 
 = C (\lambda_k ^ D + 1)^2\,.$$ 

\smallskip
 
Denoting by $\rho ^ D$ the fundamental gap of the Dirichlet Laplacian, namely the difference between the first two Dirichlet Laplacian eigenvalues, 
by \cite[Lemma 1 and Theorem 7]{Fil1} we have that
$$\| u ^ \b - u ^ D \| _{ L ^ 2 (\Omega)}  \leq \frac{2 \sqrt 2}{\rho ^ D}  {\| A^ \b - A ^ D\| } \,,$$
 and hence, thanks to the lower bound for $\rho ^ D$ proved in \cite{AnCl}, we have 
\[
\| u ^ \b - u ^ D \| _{ L ^ 2 (\Omega)}  \leq \frac{2 \sqrt 2 d(\Omega) ^ 2}{3 \pi ^ 2 }  {\| A^ \b - A ^ D\| } \,,
\]
which in view of the inequality \eqref{f:operators} proves \eqref{f:convL2_due} with 
\[
\Lambda = \frac{2 \sqrt 2 d(\Omega) ^ 2}{3 \pi ^ 2 } \, C \,.
\]

 \medskip

(ii)  Since in this case $( u ^ \b - u ^ D)$ is a harmonic function in $\Omega$, via the maximum principle we obtain 
\begin{equation}\label{f:maxp}0 \leq \sup _{\partial \Omega } u ^ \b \leq \sup _\Omega ( u ^ \b - u ^ D ) 
= \sup _{\partial \Omega }   ( u ^ \b - u ^ D )  = \frac{1}{\b}  \sup _{\partial \Omega} 
 \frac{\partial u ^ \b } {\partial \nu} \,. 
 \end{equation}
 Next observe that  the assumption $\Omega \in \reg{h} $ with $[h - \frac{N}{2}] \geq 1$ ensures the continuity of the embedding of $H ^h (\Omega)$ into $\mathcal C ^ 1 (\overline \Omega)$. Hence, by  Theorem \ref{l:bound} (ii), \eqref{f:maxp} and the H\"older inequality, we infer that $\|  u ^ \b - u ^ D \| _{L ^ 2 (\Omega)} \leq \frac{\Lambda '}{\beta}$, for a positive constant $\Lambda '$ (depending {\it only} on the quantities indicated in the statement). 
\end{proof}

\bigskip
As a by-product of Proposition \ref{l:cL2}  we obtain the following lower bound for the fundamental gap of the Robin-Laplacian:

\begin{corollary}\label{thm:gap}
  Let $\Omega\subset \R ^N$ be an open convex  domain of class $\mathcal C ^ {1,1}$. Then
\[
 \lambda _2 ^ \b(\Omega) - \lambda _1 ^ \b (\Omega) \geq \frac{3 \pi ^ 2 }{d(\Omega) ^ 2}  - \frac{1}{\beta}
  \sqrt 6  ( 1  + 2\sqrt{N}\, \kappa_{max})   (\lambda_2 ^ D + 1)^2
\,.
\]
\end{corollary}

\begin{proof}
By the lower bound in \cite{AnCl} for the Dirichlet fundamental gap, and the inequality $\lambda _ 1 ^\b \leq \lambda _ 1 ^ D$, we have
$$\begin{array}{ll} \lambda _2 ^ \b(\Omega) - \lambda _1 ^ \b (\Omega)  
& \displaystyle = 
\Big [ \lambda _2 ^D(\Omega) - \lambda _1 ^ D  (\Omega) \Big ] 
+ \Big [ \lambda _2 ^\b(\Omega) - \lambda _2 ^ D  (\Omega) \Big ] 
+ \Big [ \lambda _1 ^D(\Omega) - \lambda _1 ^ \b  (\Omega) \Big ] 
\\ 
\noalign{\medskip}
&  \displaystyle 
\geq 
\frac{3 \pi ^ 2 }{d(\Omega) ^ 2}
+ \Big [ \lambda _2 ^\b(\Omega) - \lambda _2 ^ D  (\Omega) \Big ] \,.
\end{array} 
$$ 
By Proposition  \ref{l:cL2} and Remark \ref{rem:explicit}, we have 
\begin{equation}\label{f:lambda2} \lambda _2 ^D(\Omega) - \lambda _2 ^ \b  (\Omega) \leq \frac{1}{\beta}   \sqrt 6  \| b \| _{ W ^ {1, \infty}}  (\lambda_2 ^ D + 1)^2\, ,    \end{equation}
where $b$ is an extension of the unit outer normal vector to $\Omega$ as given by  Lemma \ref{l:ext}.

In turn, explicit bounds for  the norm $\| b \| _{ W ^ {1, \infty}}$ as defined in Remark \ref{rem:explicit} in terms of the principal curvatures of $\Omega$ can be easily constructed by using a smooth cut-off of 
the distance function from the boundary. 
Specifically, 
in the same setting of the proof of Lemma~\ref{l:ext},
let $\phi\colon [0,+\infty)\to\R$ be the $C^{1,1}$ function
such that $\phi(0) = 0$, $\phi' = 1$ in $[0, \varepsilon/2]$,
$\phi'=0$ in $[\varepsilon,+\infty)$, $\phi'$ affine in $[\varepsilon/2, \varepsilon]$, and let
$b = \nabla(\phi( d ))$,
where $d$ denotes the distance function from the boundary.
Then $b\in W ^ {1, \infty}(\Omega; \R ^N)$, and
\[
Db = \phi''(d) \, \nabla d \otimes \nabla d + \phi'(d)\, D^2 d
\qquad \text{a.e.\ in}\ \Omega.
\]
Using the explicit form of $D^2 d$ (see \cite[Lemma~14.17]{GT})
and choosing $\varepsilon = 1/ \kappa_{max}$ we get the estimate
\[
\norm{b}_{ W ^ {1, \infty}} 
= \norm{ |b| }_{L^\infty} + 
\norm{ |D b| }_{L^\infty}\leq
1  + 2\sqrt{N}\, \kappa_{max}\,,
\]
which, combined with \eqref{f:lambda2}, achieves the proof. 
\end{proof}

\begin{proof}[Proof of Theorem \ref{t:convtoD}] 
(i) 
We prove the estimate  \eqref{f:convHk} by induction on $\hi$.  Throughout the proof,  $M$ denotes a constant which may vary in each inequality; 
moreover, let us pinpoint that,  in each occurrence, $M$ may depend increasingly on 
$d (\Omega)$, and $\lambda ^ D (\Omega)$:  we prefer to  omit this dependence for simplicity of writing, and we indicate just the dependence on $\delta _m := \delta _m (\Omega)$ (since
the choice of the parameter $m$ changes during the inductive process).   We also omit to indicate the fact that the latter dependence is always increasing.

For $\hi=0$, the estimate  \eqref{f:convHk} holds true by Proposition \ref{l:cL2} (i).
Let us prove it for $\hi=1$. The function $w^\b:= u ^D - u ^ \b $ satisfies the boundary value problem 
\begin{equation}\label{f:nonhdir}
\begin{cases}
- \Delta w^\b  + w  ^\b= h ^ \b & \text { in } \Omega 
\\ 
w ^ \b = \frac{1}{\b} \frac{\partial u ^ \b } {\partial \nu}  & \text{ on } \partial \Omega \,,
\end{cases}
\end{equation} 
where 
\begin{equation}\label{f:defh} h ^ \b := (\lambda ^ \b + 1 ) w ^ \b + (\lambda ^ D - \lambda ^ \b) u ^ D \,.
\end{equation} 
By Proposition \ref{l:cL2},
\begin{equation}\label{f:estih}
\| h ^ \b  \| _{L^ 2 (\Omega)} \leq \frac{M (\delta _ 2 )}{\b} \,.
\end{equation}
Then
$$
\| w ^ \b  \| _{ H ^ 1 (\Omega)}  \leq   \| h ^ \b \| _{ L ^ 2 (\Omega)} + \frac{1}{\b}    \|u^ \b     \|  _{H ^2(\Omega)} 
 \leq \frac{M (\delta _ 2 )}{\b} \,;
$$ 
here the first inequality is obtained by arguing as in the proof inequality \eqref{f:step0} (namely via the estimate
\eqref{f:base} applied to the non-homogeneous Dirichlet problem  \eqref{f:nonhdir},  
after extending $\nu$ to a vector field $b \in W ^ {1, \infty}(\Omega; \R ^N)$, with $\| b \| _{W ^ { 1, \infty}( \Omega; \R ^N)} \leq  K(\delta _ 2 (\Omega) _\uparrow)$),
while the second inequality follows from \eqref{f:estih} and the estimate \eqref{f:Hk} in Theorem \ref{l:bound} (applied with $\hi = 2)$.

\medskip
To perform the inductive step, let $\hi \geq 1$ be a fixed integer
and assume that \eqref{f:convHk} holds true for $\hi$.
Under the assumption that $\Omega$ is of class $C^{\hi + 3}$,
let us show that \eqref{f:convHk} holds true also for $\hi+1$.

From Lemma~\ref{l:ext} there exists 
a vector field $b\in W ^ {\hi +2, \infty} ( \Omega; \R ^N)$ 
with $b = \nu$ on $\partial \Omega$ and 
$\|b\|_{W^{m+2, \infty}( \Omega; \R ^N) } \leq K (\delta _{\hi + 3} (\Omega)_\uparrow)$.

The function 
\[
\widetilde w ^ \b := w ^ \b - \frac{1}{\b} \, b \cdot \nabla u ^ \b 
\]
satisfies the boundary value problem 
\begin{equation}\label{f:dir}\begin{cases}
- \Delta \widetilde w^\b  + \widetilde w  ^\b= \widetilde h ^ \b & \text { in } \Omega 
\\ 
\widetilde w ^ \b = 0  & \text{ on } \partial \Omega \,,
\end{cases}
\end{equation}
where 
\begin{equation}\label{f:defhtilde}
\widetilde h ^ \b  = h ^ \b + \frac{1}{\beta} \Big [( \Delta b - ( 1 + \lambda ^\b) b )  \cdot \nabla u ^ \b + 2 \nabla ^ 2 u ^ \b \cdot \nabla b  \Big ] \,.
\end{equation}

By the inductive assumption, we know that
\[
\| w^\b \|_{ H^\hi  (\Omega)}
\leq \frac{M (\delta_{m+2} ) }{\b }\,,
\]
hence, by the definition \eqref{f:defh} of $h^\b$, 
the estimate $\| u^D\|_{H^m} \leq C(\delta_m(\Omega))$
(see Remark \ref{r:proof-dir} (ii)), and
\eqref{f:convL2_uno}, it holds that
\[
\| h^\b \|_{ H^\hi  (\Omega)}
\leq \frac{M (\delta_{m+2} ) }{\b }\,.
\]
Recalling that $\|b\|_{W^{m+2, \infty}( \Omega; \R ^N) } \leq K (\delta _{\hi + 3} (\Omega)_\uparrow)$,
from \eqref{f:defhtilde} we obtain the estimate
\[
\| \widetilde{h}^\b \|_{ H ^ \hi (\Omega)}  \leq \frac{M(\delta _{m+3} )}{\b }  \,.
\]
Since $\widetilde{w}^\b$ is a solution to \eqref{f:dir},
it follows that
\[
\| \widetilde{w}^\b \|_{ H^{\hi+1}  (\Omega)}
\leq 
\| \widetilde{h}^\b \|_{ H ^ \hi (\Omega)} 
\leq
\frac{M (\delta_{m+3} ) }{\b }\,.
\]
The proof of the inductive step is then
completed observing that
\[
\|{w}^\b \|_{ H^{\hi+1}  (\Omega)}
\leq
\| \widetilde{w}^\b \|_{ H^{\hi+1}  (\Omega)}
+ \frac{1}{\b}\,
\| b \cdot \nabla u^\b \|_{ H^{\hi+1}  (\Omega)}
\leq
\frac{M (\delta_{m+3} ) }{\b }\,.
\]

\smallskip
Finally, under the assumption $[\hi-\frac{N}{2}] \geq 2$, the estimate \eqref{f:convC2}  follows from \eqref{f:convHk}, combined with classical Morrey-Sobolev embedding theorem (see {\it e.g.}\ \cite[Section 9.3]{Brezis}).

\medskip

 (ii) The proof is analogous to the case of the ground state. 
However, for the sake of completeness,  we sketch it  in order  to enlighten the main changes.  
 We denote by $M'$ a constant which may vary in each inequality, and 
may depend increasingly on 
$|\Omega|$, and decreasingly on $\lambda ^ D (\Omega)$. We only denote the dependence of $M'$  on $\delta _m := \delta _m (\Omega)$, which is tacitly meant to be increasing.

We argue by induction on $\hi$. 
For $\hi=0$, the statement holds by  Proposition  \ref{l:cL2} (ii).
Let us prove it for $\hi=1$. The function $w^\b := u ^D - u ^ \b $ satisfies the boundary value problem 
\[
\begin{cases}
 \Delta w^\b  = 0 & \text { in } \Omega 
\\ 
w ^ \b = \frac{1}{\b} \frac{\partial u ^ \b } {\partial \nu}  & \text{ on } \partial \Omega \,,
\end{cases}
\]
so that
$$
\| w ^\b \| _{ H ^ 1 (\Omega)}  \leq  \frac{M' (\delta_2)}{\b}    \|u^ \b    \|  _{H ^2(\Omega)} 
 \leq \frac{M' (\delta_2)}{\b} \,.
$$

Assume now that \eqref{f:convHk} holds true for a fixed integer $\hi \in \N$, $\hi \geq 1$, and  let $b\in W ^ {\hi+2, \infty} ( \Omega; \R ^N)$  
be a vector field such that $b = \nu$ on $\partial \Omega$ and $\|b \| _{W^ { \hi + 2,\infty} (\Omega; \R ^N)} \leq K(\delta _{\hi + 3} (\Omega) _\uparrow)$.  The function 
\[
\widetilde w ^ \b := w ^ \b - \frac{1}{\b} \, b \cdot \nabla u ^ \b
\]
satisfies 
\[
\begin{cases}
- \Delta \widetilde w^\b  + \widetilde w  ^\b= \widetilde h ^ \b & \text { in } \Omega 
\\ 
\widetilde w ^ \b = 0  & \text{ on } \partial \Omega \,,
\end{cases}
\]
where 
\[
\widetilde h ^ \b  = w ^ \b + \frac{1}{\beta} \Big [( \Delta b - b )  \cdot \nabla u ^ \b + 2 \nabla ^ 2 u ^ \b \cdot \nabla b  \Big ] \,.
\]
Then we have 
$$ \begin{array}{ll} \displaystyle \| w ^ \b \| _{ H ^ \hi (\Omega)}    \leq \frac{M' (\delta_{\hi +2})}{\b } \ & \displaystyle \Rightarrow \  \| \widetilde h ^ \b \| _{ H ^ {\hi}  (\Omega)}  \leq \frac{M' (\delta_{\hi+3})}{\b } 
  \ \Rightarrow \  \\ 
\noalign{\bigskip} 
& \displaystyle  \Rightarrow \| \widetilde w ^ \b \| _{ H ^ { \hi+1} (\Omega)} 
 \leq \frac{M'(\delta_{\hi+3})}{\b } \ \Rightarrow \  \|  w ^ \b \| _{ H ^ { \hi +1} (\Omega)} \leq \frac{M'(\delta_{\hi+3})}{\b }
\,, \end{array} $$ 
where each implication can be justified similarly as done in the case of the ground state. 
\end{proof}

\section{Strict convexity near the boundary}\label{sec:kor} 

In this section we focus on the study of convexity near the boundary for the function $v ^ \b = - \log ( u ^\b)$ ($u ^ \b$ being the 
Robin ground state) and for the function $v^\b = - (u ^ \b  ) ^ {1/2}$ ($u ^\b$ being the Robin torsion function). 
Let us point out that 
both functions are defined in $\overline{\Omega}$.
Specifically, as soon as $\Omega$ satisfies an interior sphere condition, 
 by Hopf's Lemma and the boundary condition,
it holds that
\begin{equation}\label{f:hopf}
\min_{\overline{\Omega}} u ^ \b = \min_{\partial\Omega} u ^ \b  > 0\,.
\end{equation}

The convexity near the boundary will be  established when the parameter $\beta$ exceeds a threshold which remains uniform in the following class of domains. 

\begin{definition}\label{def:class} For $m \in \N$, $m \geq 2$, and positive constants $\overline \delta$, $\overline d$, and $\underline \kappa$, we set 
$$
\mathcal A _\hi (  \overline \delta, \overline d, \underline \kappa):= \Big \{ \text{ open convex domains } \Omega \in \mathcal C ^ \hi \ :\  \delta _\hi (\Omega) \leq \overline \delta\, , \ 
d (\Omega) \leq \overline d\, ,\  \kappa _{min} (\Omega) \geq \underline \kappa \Big \}\,.
$$
\end{definition}

We shall exploit the following elementary lemma. 
\begin{lemma}\label{l:elem}
Let $\Omega$ be an open bounded convex set of class $\mathcal C ^ 2$, let $x_0$ be a given point in $\partial \Omega$, and   
let $\{e_1, \dots, e _{N-1}\}$  and $\{\kappa _1, \dots, \kappa _{N-1} \}$ be respectively principal directions and principal curvatures for $\partial \Omega$ at $x_0$.  
For any function $u \in \mathcal C ^ 2 (\overline \Omega)$ vanishing on $\partial \Omega$,
it holds that
\begin{equation}\label{f:secondD}
u _{ii}(x_0) = - |\nabla u (x_0)|  \kappa _i \qquad \forall \, i = 1, \dots, N-1\,.
\end{equation}

\end{lemma} 
\begin{proof}  
In the system of principal coordinates $\{e_1, \dots, e _{N-1}\}$, we can represent $\partial \Omega$ near $x_0=0$ as the graph of a function $\varphi$  satisfying $\varphi (0) = 0$,  $\nabla \varphi (0) = 0$. 
The Dirichlet condition satisfied by $u$ along $\partial \Omega$ reads
$u (x', \varphi (x')) = 0$. 
Differentiating twice such equality at $0$, and taking into account that $\nabla \varphi (0) = 0$ and $\partial _{ii} \varphi (0) = \kappa _i$, 
leads to the identities \eqref{f:secondD}. 
\end{proof}

\begin{proposition}
\label{p:Kor} 
Let $\Omega\subset\R^N$ be an open uniformly convex domain of class 
$\reg{\hi}$, with $[\hi -\frac{ N}{2}]\geq 4$.  

\smallskip
(i)  Let $u ^ \b$ be the Robin ground state of $\Omega$, and let $v ^ \b:=  - \log ( u ^\b)$. 

There exists a positive threshold $ \b^*= \b^* (\dk(\Omega)_\uparrow, d (\Omega) _\uparrow, {\kappa  _{min} (\Omega)} _\downarrow , \lambda ^ D (\Omega)_\uparrow, q  (\Omega)_\downarrow )$ such that, for $\b \geq  \b^* $, 
the Hessian matrix of $v ^ \b$ is positive definite in an inner tubular neighbourhood of $\partial \Omega$; 
in particular, for domains $\Omega\in \mathcal A _\hi  (\overline \delta, \overline d, \underline \kappa)$, there exists a uniform
 threshold  $ \beta ^* = \beta ^*(\overline \delta, \overline d, \underline \kappa)$ such that the above property holds true.

\smallskip 
(ii)  Let $u ^ \b$ be the Robin torsion function of $\Omega$, and let $v ^ \b :=- (u ^ \b) ^ {1/2}$. 

There exists a positive threshold $ \b^{**} = \beta ^ {**}  (\dk(\Omega)_\uparrow , |\Omega |_\uparrow , {\kappa _{min}  (\Omega)} _\downarrow, \lambda ^ D (\Omega)_\downarrow, p  (\Omega)_\downarrow  )$ such that, for $\b \geq  \b^{**} $, 
the Hessian matrix of $v ^ \b$ is positive definite in an inner tubular neighbourhood of $\partial \Omega$; 
 in particular, for domains $\Omega\in \mathcal A _\hi (\overline \delta, \overline d, \underline \kappa)$, there exists a uniform
 threshold  $ \beta ^{**} = \beta ^{**}(\overline \delta, \overline d, \underline \kappa)$ such that the above property holds true. 
\end{proposition}

\smallskip
\begin{remark}
Let us point out that, since under domain dilation we have
$\lambda^ \beta ( t \Omega ) = t ^{- 2} \lambda ^ { t \beta}  (\Omega )$
and the first Robin ground state scales accordingly, 
the concavity threshold $\beta ^*$ must be homogeneous of degree $-1$ under domain dilation, i.e.
$\beta ^* (t \Omega) = t ^ { -1} \beta ^* (\Omega)$.  
Nevertheless, for the sake of simplicity, we do not trace such homogeneity in our estimates. 
A similar observation is valid for the Robin torsion case. 
\end{remark}

Before starting the proof of Proposition \ref{p:Kor}, it is useful to recall some results from the previous sections. 
Let $u ^ \b, u ^ D$ be respectively the Robin and Dirichlet ground states of $\Omega$.
Since we are assuming $\Omega \in \reg{\hi}$, with $[\hi -\frac{ N}{2}]\geq 4$, by Theorem~\ref{l:bound}, Remark \ref{r:proof-dir} (ii), 
and Theorem~\ref{t:convtoD}, we have that
\begin{eqnarray}
& \| u ^ \b  \| _{ \mathcal C ^ {2 , \theta} (\overline \Omega)}  \leq   {C}
 \qquad \text{ with }  C= C ( \dk(\Omega)_\uparrow, \lambda ^ D (\Omega)_\uparrow)\,,
 & \label{f:C2_rewind} 
 \\
 & \| u ^ D  \| _{ \mathcal C ^ {2 , \theta} (\overline \Omega)}  \leq   {\widetilde C}
 \qquad \text{ with } \widetilde C= \widetilde  C (\dk(\Omega)_\uparrow, \lambda ^ D (\Omega)_\uparrow)\,,
& \label{f:C2_rewind'} 
\\
&
\displaystyle \| u ^ \b - u ^ D \| _{ \mathcal C ^ {2, \theta} (\overline \Omega)}  \leq   \frac{M}{\b }\,, \qquad \text{ with } M= M ( \dk(\Omega)_\uparrow, d (\Omega) _\uparrow, \lambda ^ D (\Omega)_\uparrow ) \,. & \label{f:convC2_rewind} 
\end{eqnarray}

\smallskip 
\begin{proof}[Proof of Proposition \ref{p:Kor}] 
(i) Let $u ^ \b, u ^ D$ be respectively the Robin and Dirichlet ground states of $\Omega$.  
Setting $v^\b := - \log ( u ^\b)$, we have to show that there exist a positive threshold $\b^*$, 
depending only on the quantities indicated in the statement,  and positive constants $\rho, \sigma$, such that 
\begin{equation}
\label{f:Kor}
\min _{\eta \in S ^ {n-1}} \pscal{\nabla ^2 v^\beta(x) \eta}{\eta} \geq \sigma \,,
\qquad
\forall \b \geq \b ^ *\,, \  \forall  x\in \overline \Omega\ \text{with}\ d(x,\partial\Omega) < \rho.
\end{equation}

To that aim we are going to prove that, given $x_0\in\partial\Omega$, 
there exist a positive threshold $ \b_0$, depending only on the quantities indicated in the statement, and positive constants $r,  \sigma$, such that \begin{equation}
\label{f:Kor'}
\min _{\eta \in S ^ {n-1}} \pscal{\nabla ^2 v^\beta(x) \eta}{\eta} \geq \sigma,
\qquad \forall \b \geq \b _0\,, \ 
\forall  x \in \overline \Omega \cap B _ r (x_0) \,.
\end{equation}
Then \eqref{f:Kor} will follow from \eqref{f:Kor'}
via a covering argument. 

A straightforward computation yields the following expression for the quadratic form associated with the Hessian of $v ^ \b$: 
\begin{equation}\label{f:qform} 
\pscal{\nabla ^2 v^\beta(x)\eta}{\eta}
= - \frac{1}{u^\beta(x)} \pscal{\nabla ^2 u^\beta(x)\eta}{\eta}
+ \frac{|\nabla u^\beta(x) \cdot \eta|^2}{u^\beta(x)^2} \qquad \forall \eta \in S ^ { N-1}\,. 
\end{equation} 

Inspired by \cite{Kor}, the idea of the proof consists in getting the estimate \eqref{f:Kor'} separately for $\eta \in T_\varepsilon(x_0)$ and for $\eta \in S ^ {N-1} \setminus T_\varepsilon(x_0)$, 
$T_\varepsilon(x_0)$ being a cone of the form
\begin{equation}\label{f:cone}
T_\varepsilon(x_0) :=
\{
\eta\in S^{N-1}:\ |\eta\cdot\nu| < \varepsilon
\}\,,
\qquad
\text{where}\
\nu := \nu(x_0) = - \frac{\nabla u^D(x_0)}{|\nabla u^D(x_0)|}\,.
\end{equation}

Let us choose properly the openness $\varepsilon$ of the cone. 

We observe that there exists $\delta > 0$ such that
\begin{equation}\label{f:uc}
 \sup _{  \tau \in S^{N-1} \cap T_{x_0} } \pscal{\nabla ^2 u^D(x_0) \tau}{\tau} < -3 \delta
\,, 
\end{equation}
where $T_{x_0} :=  \{\eta\in S^{N-1}:\ \eta\cdot\nu = 0\}$
is the tangent space to $\partial\Omega$ at $x_0$.

This follows from Lemma \ref{l:elem} applied to $u ^ D$, taking into account that $\Omega$ is uniformly convex and, by Hopf's lemma, the minimum of $|\nabla u ^ D|$ along the boundary is strictly positive. 

Notice that $\delta= \delta (q (\Omega)_\uparrow , {\kappa _{min} (\Omega)}_\uparrow)$, where we recall that
\[
q(\Omega):= \min _{x \in \partial \Omega}  |\nabla u ^ D (x)| \qquad \text{ and } \qquad   \kappa _{min} (\Omega) :=  \min _{x \in \partial \Omega, i = 1, \dots, N-1}    \kappa _i (x)\,.
\] 

We can assume without loss of generality that $\delta < 2\sqrt{2}\, C$,
where $C$ is the constant appearing in \eqref{f:C2_rewind}. 
Let us fix $\varepsilon \in (0,1)$ satisfying
\begin{equation}\label{f:epsilon}
2 \sqrt 2 C \sqrt{1 - \sqrt{1-\varepsilon^2}} = \delta\, , 
\qquad\text{i.e.}\qquad
\varepsilon = \sqrt{1 - \left(1- \frac{\delta^2}{8 C^2}\right)^2}
\,,
\end{equation}
and let us consider the cone  $T_{\varepsilon} (x_0)$ defined in \eqref{f:cone}.  
Notice that $\varepsilon$
is an increasing function of the ratio $\delta / C$, so that
$\varepsilon = \varepsilon (\delta_\uparrow, C_\downarrow)$.

\medskip
-- {\it Estimate for $\eta\in T_\varepsilon(x_0)$}. 
By \eqref{f:C2_rewind'}, \eqref{f:convC2_rewind} and \eqref{f:uc}, we can choose a positive constant $r_1$ and 
a positive threshold 
$\beta_1= \beta _1
(\dk(\Omega)_\uparrow, d (\Omega) _\uparrow, \lambda ^ D (\Omega)_\uparrow )
$  such that
\begin{align} 
& u ^ \b (x) < 1\, ,     & \forall \beta \geq \beta _1 \, , \ 
\forall x\in  \overline{\Omega} \cap B_{r_1}(x_0), 
\label{f:value} 
\\
& \sup _{\tau \in  S^{N-1} \cap T_{x_0} } 
\pscal{\nabla^2 u^\beta(x) \tau}{\tau} < -2 \delta \, ,  
& \forall \beta \geq \beta _1 \, , \ 
\forall x\in  \overline{\Omega} \cap B_{r_1}(x_0). 
\label{f:value2}  
\end{align}

We claim that
\begin{equation}
\label{f:estitan}
\inf_{\eta \in T _\varepsilon (x_0) } \pscal{\nabla ^2 v^\beta(x)\eta}{\eta}
\geq \delta,
\qquad \forall \b \geq \b _1\, , \ 
\forall x\in B_{r_1}(x_0)\,.
\end{equation}

Let us prove the claim. 
From \eqref{f:qform} we see that 
\[
\pscal{\nabla ^2 v^\beta(x)\eta}{\eta}
\geq
- \frac{1}{u^\beta(x)} \pscal{\nabla ^2 u^\beta(x)\eta}{\eta} \qquad \forall \eta \in S ^ { N-1} \, .
\]

Hence, 
$$\inf_{\eta \in T _\varepsilon (x_0) } \pscal{\nabla ^2 v^\beta(x)\eta}{\eta}\geq
- \frac{1}{u^\beta(x)} \sup_{\eta \in T _\varepsilon (x_0) } \pscal{\nabla ^2 u^\beta(x)\eta}{\eta}  \, , 
$$ 
and in view of \eqref{f:value} we are reduced to show that 
\begin{equation}\label{f:cono} \sup _{\eta \in  T_\varepsilon (x_0) } 
\pscal{\nabla^2 u^\beta(x) \eta}{\eta} \leq - \delta \, ,  \qquad   \forall \beta \geq \beta _1 \, , \ 
\forall x\in  \overline{\Omega} \cap B_{r_1}(x_0).
\end{equation}

Let $\eta$ be arbitrarily fixed in $T_\varepsilon(x_0)$, 
and let $\eta_t$ be the unit tangent vector  $\widetilde \eta_t / |\widetilde \eta_t|$, where
$ \widetilde \eta_t$ is the projection of $\eta$ onto $T_{x_0}$. 
By \eqref{f:value2}, we have that
\begin{equation}\label{f:proj}
\pscal{\nabla^2 u^\beta(x) \eta_t}{\eta_t} < -2 \delta \, ,  
\qquad
\forall \beta \geq \beta _1 \, , \ 
\forall x\in  \overline{\Omega} \cap B_{r_1}(x_0).
\end{equation}
On the other hand, 
\begin{equation}\label{f:tang}
\eta\in T_\varepsilon(x_0)
\quad\Longleftrightarrow\quad
|\eta \cdot \eta_t| > \sqrt{1-\varepsilon^2}
\quad\Longleftrightarrow\quad
|\eta-\eta_t|  ^2 < 2(1- \sqrt{1-\varepsilon^2}).
\end{equation}
Then, by \eqref{f:C2_rewind}, \eqref{f:tang} and \eqref{f:epsilon}, for every $x \in \overline \Omega$ and every $\beta >1$ we have 
\begin{equation}\label{f:2term} 
\begin{array}{ll} 
\displaystyle \left|\pscal{\nabla ^2 u^\beta(x)\eta}{\eta} - \pscal{\nabla ^2 u^\beta(x)\eta_t}{\eta_t}\right|
& \leq
2 |\nabla ^2 u^\beta(x)|\, |\eta - \eta_t|
\\ \noalign{\medskip}
& \displaystyle \leq
2 \sqrt 2 C \sqrt{1- \sqrt{1-\varepsilon^2}} = \delta.
\end{array} 
\end{equation} 

The required inequality \eqref{f:cono} follows from \eqref{f:proj} and \eqref{f:2term}.

\medskip
-- {\it Estimate for $\eta\in S ^{N-1} \setminus T_\varepsilon(x_0)$}. 
Let us choose first  $\b_2$ and then $r_2>0$ such that 

\[
\b _ 2 > \frac{2M} { \varepsilon \, q (\Omega)} \, , \qquad 
  r_2  
\leq \frac{1}{\widetilde C } \Big [\frac{\varepsilon q (\Omega) }{2} - \frac{M}{\beta_2} \Big]\,,
\]
with 
\[
\begin{array}{ll} \beta _ 2 &  = \beta _2 (M_\uparrow , \varepsilon _\downarrow, q(\Omega)_\downarrow) 
= 
\beta _2 (M_\uparrow , C_\uparrow , \delta _\downarrow, q(\Omega)_\downarrow) = 
\beta _ 2 (M_\uparrow , C_\uparrow , q  (\Omega)_\downarrow , {\kappa_{min}  (\Omega)} _\downarrow)  
\\  \noalign{\medskip} 
& = \beta _2 ( \dk(\Omega)_\uparrow, d (\Omega) _\uparrow, \lambda ^ D (\Omega)_\uparrow, q  (\Omega)_\downarrow , {\kappa _{min}  (\Omega) }_\downarrow ) \,.
\end{array}
\] 
Let $\eta$ be arbitrarily fixed in $S ^{N-1} \setminus T_\varepsilon(x_0)$. For $\b \geq \b _ 2$ and $x\in \overline \Omega \cap B_{r_2}(x_0)$, 
we have that 
\[
\begin{split}
|\nabla u^\beta(x) \cdot \eta| & \geq
|\nabla u^D(x_0) \cdot \eta| 
-|\nabla u^D(x) - \nabla u^\beta(x)|
- |\nabla u^D(x_0) - \nabla u^D(x)| 
\\ & \geq
 \varepsilon q (\Omega) - \frac{M}{\beta}
 - \widetilde C |x-x_0| \, , 
\end{split}
\]
so that
\[ 
|\nabla u^\beta(x) \cdot \eta|  \geq \frac{\varepsilon \, q (\Omega)}{2}\qquad \forall \beta \geq \beta _2\, , \ 
\forall x\in \overline \Omega \cap B_{r_2}(x_0)\, ,
\]
and hence, recalling \eqref{f:qform}, 
\[
\inf_{\eta \in S ^ { N-1} \setminus T _\varepsilon (x_0) } \pscal{\nabla ^2 v^\beta(x)\eta}{\eta}
\geq
- C \frac{1}{u^\beta(x)} + 
\frac{ \varepsilon^2 q ^ 2 (\Omega)}{4} \frac{1}{ u^\beta(x)^2},
\qquad \forall \beta \geq \beta _2\, , 
\forall x\in B_{r_2}(x_0)\, . 
\]
It is readily checked that 
there exists $s = s (C_\downarrow , q (\Omega)_\uparrow, \varepsilon _\uparrow) > 0$ such that
\[
- C \frac{1}{u^\beta(x)} + 
\frac{ \varepsilon^2 q ^ 2 (\Omega)}{4} \frac{1}{ u^\beta(x)^2} \geq    \frac{C ^ 2}{\varepsilon ^ 2 q ^ 2 (\Omega)} \qquad \text{ for  }
u _\b (x) \leq s\,.
\] 
By \eqref{f:C2_rewind'}-\eqref{f:convC2_rewind}, we can choose $\beta _3$ and $r _3$ such that
\[
u _\b (x) \leq s \qquad \forall \beta \geq \beta _3\, ,\ x \in \overline \Omega\cap B _ {r_3} (x_0)\, ,
\] 
with 
\[
\begin{split} 
\beta _3 & =  \beta _ 3 (s_\downarrow,  \dk(\Omega)_\uparrow, d (\Omega) _\uparrow, \lambda ^ D (\Omega)_\uparrow )
=  \beta _ 3 (\dk(\Omega)_\uparrow, d (\Omega) _\uparrow, \lambda ^ D (\Omega)_\uparrow, q(\Omega) _\downarrow, \varepsilon _\downarrow   )  
\\
& =  \beta _ 3 (\dk(\Omega)_\uparrow, d (\Omega) _\uparrow, \lambda ^ D (\Omega)_\uparrow, q(\Omega) _\downarrow, {\kappa _{min}  (\Omega) }_\downarrow   )  
\,. 
\end{split}
\]
We conclude that 
\begin{equation}\label{f:estinontan} \inf_{\eta \in S ^ { N-1} \setminus T _\varepsilon (x_0) } \pscal{\nabla ^2 v^\beta(x)\eta}{\eta}
\geq  \frac{C ^ 2}{\varepsilon ^ 2 q ^ 2 (\Omega)} \, ,
\qquad \forall \beta \geq \beta _2 \vee \beta _3 \, , 
\forall x\in \overline \Omega \cap B_{r_2 \wedge r _3}(x_0)\, . 
\end{equation}

By combining \eqref{f:estitan} and \eqref{f:estinontan},
 we see that the required property \eqref{f:Kor'} holds true with 
$\sigma :=\delta \wedge\frac{C ^ 2}{\varepsilon ^ 2 q ^ 2 (\Omega)}$,   
$r := r_1 \wedge r _2 \wedge r _3$, and
\[
\beta _0  := \beta _1 \vee \beta _2 \vee \beta_ 3 = \beta _0( \dk(\Omega)_\uparrow, d (\Omega) _\uparrow, \lambda ^ D (\Omega)_\uparrow, q(\Omega) _\downarrow, {\kappa_{min}   (\Omega) }_\downarrow   ) \,.
\]  

\medskip

Let us finally check the existence of a uniform threshold $\beta ^*$ independent of $\Omega$  within the class
$\mathcal A _\hi (\overline \delta, \overline d, \underline \kappa)$. 
 By the monotone dependence of $\beta ^*$ by the quantities indicated in the statement, it is enough to show that,  on the class  $\mathcal A _\hi (\overline \delta, \overline d, \underline \kappa)$, we have that 
$ \lambda ^ D (\Omega)$ is uniformly bounded from above and $q (\Omega)$ is uniformly bounded from below. 
Notice that domains in the class  $\mathcal A _\hi (\overline \delta, \overline d, \underline \kappa)$  satisfy a uniform interior sphere condition of radius $\rho \geq \underline{\kappa}^{-1}$. 
This gives immediately the required upper bound for $\lambda ^ D(\Omega)$ thanks to the monotonicity of $\lambda ^ D(\cdot)$ by inclusions. 
On the other hand, the required lower bound for $q (\Omega)$  follows from Theorem \ref{t:Jerison2} (see Section \ref{sec:appendix}), 
since the right-hand side of the estimate~\eqref{f:estiJ2} is monotone decreasing in $\lambda^ D$, $d$, 
and in the radius $\rho$ of the uniform interior sphere condition (which are bounded from above in our class) 
and increasing in the quantity $\sigma$  defined  in~\eqref{f:sigma} (which is bounded from below in our class).
\medskip

(ii) Let $u ^ \b$ be the Robin torsion function of $\Omega$.   
The quadratic form associated with the Hessian of $v ^ \b :=- (u ^ \b) ^ {1/2}$ is written as
\[
\pscal{\nabla ^2 v^\beta(x)\eta}{\eta}
= - \frac{1}{2 (u^\beta (x))^ {1/2} } \pscal{\nabla ^2 u^\beta(x)\eta}{\eta}
+ \frac{|\nabla u^\beta(x) \cdot \eta|^2}{4 (u^\beta(x)) ^ {3/2}} \qquad \forall \eta \in S ^ { N-1}\,. 
\]

Then the proof proceeds in parallel to the case of the ground state, invoking at each stage the results from Section \ref{sec:estimates}
(analogous to \eqref{f:C2_rewind}-\eqref{f:C2_rewind'}-\eqref{f:convC2_rewind}),
which involve the Robin torsion function in place of the Robin ground state.

Concerning the existence of a uniform  threshold $\beta ^{**}$ independent of $\Omega$  within the class
$\mathcal A _\hi (\overline \delta, \overline d, \underline \kappa)$, 
by the monotone dependence of $\beta ^{**}$ by the quantities indicated in the statement, we only have check that,  on the class  $\mathcal A _\hi (\overline \delta, \overline d, \underline \kappa)$,
$ |\Omega|$ is uniformly bounded from above and $\lambda ^ D(\Omega)$, $p (\Omega)$ are uniformly bounded from below. 
The upper bound  from above for $|\Omega|$ is immediate in view of the boundedness  from above for  the diameter. 
The lower bound for $\lambda ^ D (\Omega)$ follows from the fact that balls minimize $\lambda ^ D (\cdot)$ in the class of sets with diameter bounded from above \cite[Theorem 2.1]{BoHeLu}.  Finally,  the lower bound for $p (\Omega)$  follows from Theorem \ref{t:torsion} (see Section \ref{sec:appendix}), since the right-hand side of the estimate \eqref{f:bandle} is monotone decreasing in $\kappa _{max}(\Omega)$ (recall that $\delta _2(\Omega)$, 
and hence $\kappa_{max} (\Omega)$, is uniformly bounded from above in the class $\mathcal A _\hi (\overline \delta, \overline d, \underline \kappa)$). 
\end{proof}

\section{ Proof of Theorems \ref{thm:concn} and \ref{thm:concn2}}\label{sec:proofs} 

We use the continuity method of
Caffarelli and Friedman \cite{CafFri}. In doing this,  we must be careful about the way 
we deform a given domain into a ball. 
To be more precise, 
let  $\Omega$ belong to $\mathcal A _\hi (\overline \delta, \overline d, \underline \kappa)$ (recall Definition \ref{def:class}). 
For $t \in [0,1]$, we consider
the following family of  convex bodies, where $+$ denotes the usual Minkowski addition in $\R ^N$:  
$$K_t := ( 1- t)\overline{ \Omega _0}  + t  \overline \Omega, \qquad \text {with } \Omega _ 0=  \Big \{ |x| <  r := \kappa_{min} (\Omega)^ {-1} \Big \} \,.$$
We claim that there exist positive constants $(\overline \delta', \overline d', \underline \kappa')$, depending only on $( \overline \delta,  \overline d, \underline \kappa)$
(in particular, {\it independent} of the parameter $t \in [0,1]$) such that
\begin{equation}\label{f:omegat}
\Omega _ t : = {\rm int} ( K _ t)  \in  \mathcal A _\hi (\overline \delta', \overline d', \underline \kappa') \qquad \forall t \in [0, 1]\,.
\end{equation}
For the fact that $\Omega _t$ is of class $\mathcal C ^\hi$ for ever $t\in [0,1]$, we refer to \cite[Proposition 5.1 (3)]{ghomi}. 
The existence of a uniform upper bound $\overline d'$  for $d(\Omega_t)$ follows immediately from the definition of Minkowski addition. 
To check the existence of a uniform upper bound $\overline \delta '$  for  $\delta _\hi (\Omega_t)$  and a uniform lower bound $\underline \kappa'$ for $\kappa _{min} (\Omega _ t)$, one can proceed as follows ({\it cf.} \cite[Section 5]{ghomi}).  Denote by $\nu _0$ and $\nu _ 1$ the Gauss map of $\overline{ \Omega _0}$ and $\overline \Omega$ respectively. 
Notice that, by the regularity assumptions made on $\Omega$ and since $\Omega _0$ is a ball, 
we have that $\overline \Omega$ and $\overline \Omega _0$ are uniformly convex bodies at least of class $\mathcal C ^2$, so that $\nu _0$ and $\nu _1$ are diffeomorphisms respectively from $\partial \Omega_0$ and $\partial \Omega$ to $S ^ {N-1}$.  By the definition of Minkowski addition we have that, for a given point $x \in \partial \Omega _ t$: 
$$x = ( 1-t ) x _0 + t x _ 1 \ \Rightarrow \ \nu _0( x_0) = \nu _ 1 ( x_1) \ \Rightarrow\  x = ( 1-t) \nu _0 ^ {-1} (\nu _ 1 ( x_1)) + t x_1\ .$$ 
Recalling that $\Omega_0$ is a ball of radius $r$, we infer that
$$x = f_ t ( x_1):= ( 1-t) r  \nu _ 1 ( x_1) + t x _ 1 \,.$$
Thus, if all the derivatives  of $\nu _1$ up to a certain order are bounded from above, the same holds true for the derivative of $f_t$, 
in terms of a constant independent of $t$, which yields the existence of the upper bound $\overline \delta '$.  Moreover, by the choice of $r$, we have 
$$df _t = (1-t ) r d \nu _1  + tI \geq   ( 1-t) r \kappa _{min} + t= I \, , $$
which yields the existence of the lower bound $\underline \kappa '$ and achieves the proof of our claim. 

\smallskip
For every $t\in [0,1]$, we denote by $\ut ^\b$ either the Robin  
ground state or the Robin torsion function of $\Omega _ t$, and we simply write $u ^ \b$ for $u ^ \b _1$. 
For every $t \in [0,1]$, $\ut ^\b$ is strictly positive ({\it cf.} \eqref{f:hopf}), and 
by Theorem \ref{l:bound} it is of class  $\mathcal C^2(\overline{\Omega_t})$. 

By the choice of the deformation (precisely, thanks to \eqref{f:omegat})  and by  Proposition~\ref{p:Kor}, there exist positive thresholds  $\beta^*$ and  $\beta ^ {**}$ 
({\it independent of $t \in [0,1]$}) such that:

\begin{itemize}
\item[(i)]  For every $\beta\geq \beta^*$ and every $t \in [0,1]$,  the Hessian matrix of the function $v_t^\b:= -\log \ut ^\b$ 
(with $\ut ^\b$ the Robin ground state) is 
positive definite in an inner $\varepsilon$-tubular neighborhood of $\partial\Omega_t$ (with $\varepsilon = \varepsilon (t)>0$).

\item[(ii)]  For every $\beta\geq \beta^{**}$ and every $t \in [0,1]$, the Hessian matrix of the function 
$v_t^\b := - ( \ut ^\b)^{1/2}$  (with $\ut ^\b$ the Robin torsion function) 
is positive definite in an inner $\varepsilon$-tubular neighborhood of $\partial\Omega_t$ (with $\varepsilon = \varepsilon (t)>0$). 
\end{itemize}

Fix now $\beta$ larger than the above threshold. 
We claim that  $\nabla ^ 2 v ^ \b$  is  positive definite in  $\overline{\Omega}$.  
Indeed, assume by contradiction that this is not the case. By  the explicit computation of $u ^ \b _0$ on the ball $\Omega_0$, we have that $\nabla ^ 2 v ^\b _0$ is positive definite in $\overline {\Omega _0}$. 
Hence, for some critical value $s \in (0, 1)$, we would have that $\nabla ^ 2 v_s ^\beta$ is positive semi-definite but not positive definite in $\overline{\Omega _s}$. 
We observe that
$v^ \b_s$ satisfies respectively 
\[
\Delta v   ^ \b  _s= \lambda ^ \b_s  + |\nabla v ^ \b  _s  |^2 \qquad \text{ and } \qquad \Delta v ^ \b  _s = - \frac{1}{v ^ \b _s} \Big [ \frac{1}{2} + |\nabla v ^ \b _s | ^ 2 \Big ]
\qquad \text{in}\ \Omega_s\,,
\]
where $\lambda ^ \b _s := \lambda ^ \b (\Omega _s)$. 
In both cases, we have that $\Delta v ^ \b _s = f ( v ^ \b _s, \nabla v ^ \b _s)$, with $1/ f (\cdot  , \nabla v ^ \b_s)$ convex. 
Then, by  \cite[Theorem 1]{KorLew}, since $\nabla ^ 2 v ^ \b _s$ is positive semidefinite in $\Omega_s$, it has constant rank in $\Omega_s$. 
But, by the choice of $\beta$, $\nabla ^ 2 v ^ \b _ s$ is positive definite in an inner $\varepsilon$-tubular neighborhood of $\partial\Omega_s$ 
 (with $\varepsilon = \varepsilon (s) >0$). Hence  $\nabla ^ 2 v ^ \b _s$  is positive definite in $\overline{\Omega _s}$, yielding a contradiction.

Let us remark that, in order to get the existence of the critical value $s$ as above, we have implicitly exploited the following continuity property of the Hessian matrix of $v ^ \b$ with respect to the parameter in the family of deformations:  
for every $s\in [0,1)$ and for every open set $A$ with
$\overline{A}\subset \Omega_s$, it holds that
\begin{equation}\label{f:sch}
\|\nabla ^2(v_t -v_s)\|_{\mathcal C ^ 2 (\overline A)} 
\to 0,
\qquad
\text{for}\ t \to s\,.
\end{equation}

The convergence in \eqref{f:sch} follows by the interior Schauder estimates and the continuity of the Robin 
ground states and the eigenvalues with respect to the parameter in the family of domains \eqref{f:omegat} we are dealing with. 
Specifically, 
if $|t-s|$ is small enough, given  open sets  $\overline{A} \subset E \subset\overline{E} \subset\Omega_s$, we have that
$\overline{E} \subset\Omega_t$.
By the classical interior Schauder estimates, there exist $\alpha\in (0,1)$
and positive constants $K, K' > 0$ such that
\[
\begin{split}
\|\nabla ^2(\ut ^ \b -u  ^ \b _s)\|_{\mathcal C ^ {2,\alpha}  (\overline A)} 
 & \leq K \left(
\|\Delta(\ut ^ \b -u_ s ^ \b )\|_{\mathcal C ^ {0,\alpha}  (\overline E)} 
 + \|\ut ^ \b  - u _ s ^ \b\|_{L^\infty(E)}
\right)
\\ & \leq
K' \left(|\lambda_t  ^ \b - \lambda_s ^ \b | +  \|\ut ^ \b  - u _s ^ \b   \|_{\mathcal C ^ {0,\alpha}  (\overline E)}  \right)\,.
\end{split}
\]
The facts that $|\lambda_t^ \b  - \lambda_s^ \b | \to 0$
and $ \|\ut ^ \b  - u _s ^ \b   \|_{\mathcal C ^ {0,\alpha}  (\overline E)}   \to 0$ 
as $t\to s$ are proved in 
\cite[Proposition~3.2]{AnClHa}
(see also \cite{BuGiTr}).
\qed

\section{Appendix: Boundary  gradient estimates for Dirichlet problems  }\label{sec:appendix}

Let $\Omega\subset\R^N$ be a convex domain. Before stating  a boundary gradient estimate for a first Dirichlet eigenfunction $u ^ D$  of $\Omega$, 
let us recall the following result, proved in \cite[Corollary~3.4]{BisLor}, about the location of a {\it hot spot}, namely
 a maximum point $\overline x$ of $u ^ D$: denoting by $r (\Omega)$ the inradius of $\Omega$, 
 there exists a universal constant $\theta \simeq 0.0833$ such that  
\begin{equation}\label{f:location}
{\rm dist}(\overline{x}, \partial\Omega) \geq {C_0} \,  {r(\Omega)}\,, 
\qquad \text{ with }\ 
C_0 := \sqrt{\frac{\theta}{\lambda^D(B^N_1)}}\,. 
\end{equation}
Here $B^N_1$ is the unit ball in $\R^N$. 
We remark that, from the explicit computation of $\lambda^D(B_1^N)$, we have that
$ 0< C_0 <1$ in any space dimension $N$.

\begin{theorem}
\label{t:Jerison2}
Let $\Omega\subset\R^N$ be a convex domain 
satisfying a uniform interior sphere condition of radius $\rho > 0$.  
Denote by $d = d (\Omega)$ the diameter of $\Omega$ and by $r = r (\Omega)$ its inradius. 
Let $u^D>0$ be a first Dirichlet Laplacian eigenfunction in $\Omega$. 

There exist  dimensional constants $C_1 >1$, $C_2>0$    such that  
\begin{equation}\label{f:estiJ1}
\min _{\partial \Omega } |\nabla u ^ D  |  \geq
\frac{C_2}{\rho}\, (\max_{\overline \Omega} u) \, 
C_1^{-d(\sqrt{\lambda^D}/2  + 2 \sqrt{N} /\sigma)}\,,
\end{equation}
where, for $C_0$ as in \eqref{f:location},
\begin{equation}\label{f:sigma}
\sigma  := \min\left\{\frac{\rho}{2}\,, \, C_0 \,  r \right\} \,.
\end{equation}

In particular, if we normalize $u ^ D$ so to have $\| u ^ D \| _{ L ^ 2 (\Omega)} = 1$, we have, for another dimensional constant $C_3>0$,
\begin{equation}\label{f:estiJ2}
\min _{\partial \Omega } |\nabla u ^ D  | \geq 
\frac{C_3}{\rho}\, d^{-N/2} \, 
C_1^{-d(\sqrt{\lambda^D}/2  + 2 \sqrt{N} /\sigma )}\,.
\end{equation}
\end{theorem}

\begin{remark} Since $\Omega$ contains a ball of radius $\rho$, estimate \eqref{f:estiJ2} implies, for another dimensional constant $C_4$,  
$$
\min _{\partial \Omega } |\nabla u ^ D  | \geq 
\frac{C_3}{\rho}\, d^{-N/2} \, 
C_4^{-d /\sigma }\,.
$$
\end{remark}

\begin{proof} 
Throughout the proof we write for brevity $\lambda$, $u$  to denote
respectively the first Dirichlet eigenvalue of $\Omega$ and a first Dirichlet eigenfunction. 

\medskip
{\it Step 1}. For $\sigma$ defined as in the statement, 
set $\Omega _ \sigma :=  \{ x \in \Omega \, :\, {\rm dist} (x, \partial \Omega) > \sigma \}$. 
We claim that
there exists a  dimensional constant $C_1 > 1$ such that 
\[
\min_ {\overline{\Omega_{\sigma} }} u  \geq 
C_1^{- d (\sqrt{\lambda} /2 + 2 \sqrt{N} / \sigma ) }   \, \max_{\overline{\Omega}} u \,. 
\]

Let $x$ be a fixed point in $\overline \Omega_{\sigma}$, and let $\overline{x}\in\Omega$ be a maximum point of $u$. 
As in \cite{JerKen}, 
we consider a Harnack chain of balls from $x$ to $\overline{x}$, that is, 
a family $B_1, \ldots, B_k$ of balls of radius $R = \frac{\sigma}{4}$,
such that:

\smallskip
\begin{itemize}
\item[$\cdot$)]   $\overline{x} \in B_1$ and $x \in B_k$; 
\smallskip

\item[$\cdot$)] ${\rm dist}(B_j, \partial\Omega) \geq  \sigma  \qquad \forall\, j=1,\ldots, k$; 
\smallskip

\item[$\cdot$)] 
$\exists \, x _j \in B_{j}\cap B_{j+1}  \qquad \forall \, j=1,\ldots, k-1$;
\smallskip

\item[$\cdot$)] $k \leq \frac{2\, d}{\sigma}$.
\end{itemize}

(Since $\Omega$ is convex, the centers of all the balls
can be chosen on the segment joining $x$ to $\overline{x}$. Note also that we are exploiting \eqref{f:location} and the definition of $\sigma$ which imply  
${\rm dist}(\overline x, \partial\Omega) \geq  \sigma$.)

\smallskip

By the Harnack inequality
(see \cite[Theorem~8.20]{GT}), for every ball $B_{4R}(y) \subset\Omega$, we have, for a dimensional constant $C_1>1$, 
\[
\sup_{B_R(y)} u \leq C_1^{\sqrt{\lambda}\, R + \sqrt{N} } \, \inf_{B_R(y)} u\,. 
\]

Applying the above estimate (with  $R = \frac{\sigma}{4}$) to the Harnack chain, we get
\[
u(x_{j}) \leq C_1 ^{\sqrt{\lambda}\, \sigma / 4 + \sqrt{N}}   \, u(x_{j+1}),
\qquad \forall j = 1, \ldots, k-1.
\]   
Hence,
\[
\max_{\overline{\Omega}} u =
u(\overline{x}) \leq  C_1 ^{k(\sqrt{\lambda}\, \sigma / 4+\sqrt{N})}   \, u(x)
\leq C_1 ^{d( \sqrt{\lambda}/ 2 + 2\sqrt{N}/\sigma)}   \, u(x) 
\,.
\]

\bigskip
{\it Step 2 (completion of the proof).}
Let $x_0\in\partial\Omega$ and let $B_\rho(z)\subset\Omega$
be such that $x_0 \in \partial B_\rho(z)$.
Without loss of generality we can assume that $z=0$.
Let
\[
w(x) := u(x) - \varepsilon\, v(x),
\quad
v(x) := e^{-\mu |x|^2} - e^{-\mu\, \rho^2},
\qquad
x\in U := {B}_\rho \setminus \overline{B}_{\rho/2}\,,
\]
where $\mu$ and $\varepsilon$ are two positive constants.
Since $u\geq 0$ and $v = 0$ on $\partial B_\rho(0)$, we clearly have
that $w \geq 0$ on $\partial B_\rho$.
Moreover, by Step~1, we have that $w \geq 0$ also on $\partial B_{\rho/2}$ provided we choose 
\[
\varepsilon =
\Big ( e^{-\mu\, \frac{\rho^2}{4}} - e^{-\mu\, \rho^2} \Big ) ^ {-1}    
(\max_{\overline{\Omega}} u) \, C_1^{- d (\sqrt \lambda /2 + 2 \sqrt N / \sigma ) }   \,.
\]

By a direct calculation
\[
\Delta v(x) = 2 \mu e ^ { - \mu |x-z| ^ 2} \big ( 2 \mu |x-z| ^ 2 - N )
\]
so that,
choosing
\[
\mu = \frac {2\, N}{\rho^2}\,,
\]
then $\Delta v \geq 0$ in $U$, so that $\Delta w \leq 0$ in $U$, and hence 
$w \geq 0$ in $U$.
In particular,
\[
\begin{split}
|\nabla u (x_0) | & \geq
\varepsilon\, |\nabla v (x_0)|  = 
2\,\varepsilon\, \mu\, \rho\, e^{-\mu\,\rho^2}
= 
\frac{C_2}{\rho} \,(\max_{\overline{\Omega}} u) \, C_1^{- d (\sqrt \lambda /2 + 2 \sqrt N / \sigma ) }    \, ,
\end{split}
\]
where $C_2$ is a positive dimensional constant.

Consider now a first Dirichlet eigenfunction $u $ normalized so to have
$\| u  \| _{ L ^ 2 (\Omega)} = 1$. Then, since
$(\max _{\overline \Omega} u ) ^ 2 |\Omega| \geq  1$, 
\eqref{f:estiJ2} follows from \eqref{f:estiJ1}.
\end{proof}

\begin{theorem}
\label{t:torsion}
Let $\Omega\subset\R^N$ be a convex domain of class $\mathcal C ^ 2$.  
Let $u^D$ be the Dirichlet torsion function in $\Omega$. 
Then 
\begin{equation}\label{f:bandle}
\min _{\partial \Omega } |\nabla u ^ D  |  \geq
(N \kappa _{max}  (\Omega)) ^ { -1} \,.
\end{equation}
\end{theorem} 

\begin{proof} 
The result is due to C.~Bandle \cite{bandle}, see also \cite[Lemma 2.2]{MaShiYe}. 
\end{proof}

\bigskip

{\bf Acknowledgments.} 
The authors would like to thank David Jerison for
kindly suggesting the proof of Theorem~\ref{t:Jerison2}.
They are also grateful to Tom ter Elst for indicating some bibliographical references. The authors
have been partially supported by the Gruppo Nazionale per l'Analisi Matematica, 
la Probabilit\`a e le loro Applicazioni (GNAMPA) of the Istituto Nazionale di Alta Matematica (INdAM). 
G.C.\ has been partially supported by Sapienza -- Ateneo 2017 Project ``Differential Models in Mathematical Physics''
and Sapienza –- Ateneo 2018 Project ``Stationary and Evolutionary Problems in Mathematical Physics and Materials Science''.

\def\cprime{$'$}

\end{document}